\newif{\ifdraft}\drafttrue
\draftfalse
\ifdraft 
\documentclass[a4paper,draft]{article}
\else
\documentclass[a4paper,final]{article}
\fi 
\usepackage{showkeys}

\usepackage[english]{babel}
\usepackage[applemac]{inputenc}

\usepackage{amsfonts}
\usepackage{amssymb}
\usepackage{amsmath}
\usepackage[arrow, matrix, curve]{xy}
\usepackage{enumerate,xspace}
\usepackage{tikz}
\usetikzlibrary{trees}
\usetikzlibrary{snakes}
\usepackage{ellipsis}
\usepackage{ifthen}

\usepackage{pdfsync}

\numberwithin{equation}{section}

\usepackage{theorem}
\usepackage{prettyref}

\newrefformat{thm}{Theorem~\ref{#1}}
\newrefformat{lm}{Lemma~\ref{#1}}
\newrefformat{dfn}{Definition~\ref{#1}}
\newrefformat{cor}{Corollary~\ref{#1}} 
\newrefformat{prop}{Proposition~\ref{#1}}
\newrefformat{sec}{Section~\ref{#1}}
\newrefformat{kap}{Chapter~\ref{#1}}
\newrefformat{bsp}{Example~\ref{#1}}
\newrefformat{rem}{Remark~\ref{#1}}
\newrefformat{fig}{Figure~\ref{#1}}
\newrefformat{eq}{Equation~\ref{#1}}

\newtheorem{theorem}{Theorem}[section]

\newtheorem{lemma}[theorem]{Lemma}
\newtheorem{proposition}[theorem]{Proposition}
\newtheorem{corollary}[theorem]{Corollary}
\newtheorem{dfn}[theorem]{Definition}

\theoremstyle{break}

\theoremstyle{normal}
\theorembodyfont{\upshape}
\newtheorem{example}[theorem]{Example}
\newtheorem{remark}[theorem]{Remark}

\newenvironment{proof}[1][Proof.]{
\begin{trivlist}
\item[\hskip \labelsep {\bfseries #1}]}{\hspace*{\fill}$\Box$\end{trivlist}
}

\newcommand{\ov}[1]{\overline{#1}}

\newcommand{\abs}[1]{\left|\mathinner{#1}\right|}

\newcommand{\set}[2]{\left\{\, \mathinner{#1}\vphantom{#2}\: \left|\: \vphantom{#1}\mathinner{#2} \right.\,\right\}}
\newcommand{\oneset}[1]{\left\{\, \mathinner{#1} \,\right\}}
\newcommand{\os}[1]{\left\{\, \mathinner{#1} \,\right\}}
\newcommand{\smallset}[1]{\left\{\mathinner{#1}\right\}}

\newcommand\lds{,\ldots ,} 

\newcommand{\sse}{\subseteq}
\newcommand{\es}{\emptyset}

\newcommand{\N}{\mathbb{N}}
\newcommand{\Z}{\mathbb{Z}}
\newcommand{\Q}{\mathbb{Q}}

\newcommand{\cG}{\mathcal{G}}

\newcommand{\cK}{\mathcal{K}}
\newcommand{\cP}{\mathcal{P}}

\newcommand{\IFF}{if and only if\xspace}
\newcommand{\homo}{homomorphism\xspace}
\renewcommand{\phi}{\varphi}
\newcommand{\ew}{\varepsilon}

\newcommand{\eps}{\varepsilon}

\newcommand{\Sg}{\Sigma}
\newcommand{\GG}{\Gamma}

\newcommand{\alp}{\alpha}
\newcommand{\bet}{\beta}
\newcommand{\gam}{\gamma}
\newcommand{\del}{\delta}
\newcommand{\sig}{\sigma}

\newcommand{\bs}{\backslash}

\newcommand\Copt{\cC_{\mathrm{opt}}}
\newcommand\RAS[2]{\overset{#1}{\underset{#2}{\Longrightarrow}}}

\newcommand{\genr}[2]{\left< \, \mathinner{#1}\vphantom{#2}\: \left|\: \vphantom{#1}\mathinner{#2} \right.\, \right>}
\newcommand{\Sym}[1]{\mathrm{Sym}({#1})}

\newcommand{\Aut}{\mathrm{Aut}}

\newcommand{\Comp}[1]{\overline{#1}}

\newcommand{\cC}{\mathcal{C}}

\newcommand{\ssnq}{\subsetneqq}
\newcommand{\sm}{\setminus}

\newenvironment{aw}{\noindent\color{red} AW }{}
\newenvironment{vd}{\noindent\color{blue} VD }{}

\begin{document}

\title{Context-Free Groups and Their Structure Trees}

\author{Volker Diekert \qquad Armin Wei\ss \\[5mm]
  Universit{\"a}t Stuttgart, FMI \\
  Universit{\"a}tsstra{\ss}e 38 \\
  D-70569 Stuttgart, Germany \\[5mm]
  \texttt{$\{$diekert$,$weiss$\}$@fmi.uni-stuttgart.de}}

\maketitle

\begin{abstract}
  \noindent
  Let $\GG$ be a connected, locally finite graph of finite tree width and
  $G$ be a group acting on it with finitely many orbits and finite node stabilizers. 
  We provide an elementary and direct construction  of a tree $T$ on which 
  $G$ acts with finitely many orbits and finite vertex stabilizers. 
  Moreover, the tree is defined directly in terms of 
  the structure tree of optimally nested cuts of $\GG$. 
  Once the tree is constructed, Bass-Serre theory yields that $G$ is virtually free. This approach simplifies the existing proofs for the fundamental result of Muller and Schupp that characterizes context-free groups 
 as f.g.~virtually free groups. Our construction avoids the explicit use of Stallings' structure theorem and it is self-contained. 

We also give a simplified proof for an important consequence of the structure tree theory by Dicks and Dunwoody which has been stated by Thomassen and Woess. It says that a f.g.{} group is accessible if and only if its Cayley
 graph  is accessible.
   \medskip

  
\end{abstract}
\section{Introduction}\label{sec:intro}
A seminal paper of Muller and Schupp \cite{ms83} showed that 
a group $G$ is context-free \IFF it is a finitely generated virtually free 
group. A group $G$ is \emph{context-free} if there is some finite set
$\Sg$ and a surjective \homo $\phi: \Sg^*\to G$ such that the associated 
group language $L_G = \phi^{-1}(1)$ is context-free in the sense of formal language theory. A group $G$ is \emph{virtually free} if it has a free subgroup of finite index. 
Finitely generated (f.g.) virtually free 
groups were the basic examples for context-free groups because  
the standard algorithm to solve their word problem runs on a deterministic 
pushdown automaton; and these automata recognize a proper subfamily of context-free languages. 
The  deep insight by Muller and Schupp is that 
the converse holds: If
$G$ is context-free, then $G$  is a finitely generated virtually free 
group. Over the past decades a wide range of other characterizations of context-free 
(or f.g.~virtually free) groups have been found
showing the importance of this class. 

The various equivalent characterizations include: 
(1) fundamental groups of finite graphs of finite groups \cite{Karrass73},  
(2) f.g.~groups having a Cayley graph which can be $k$-triangulated \cite{ms83}
(3) f.g.~groups having a Cayley graph of finite treewidth \cite{KuskeL05},
(4) universal groups of finite pregroups \cite{Rimlinger87a}, 
(5) groups having a finite presentation by some geodesic string rewriting system \cite{GilHHR07}, and
(6) f.g.~groups having a Cayley graph with decidable monadic second-order theory \cite{KuskeL05}. 
For some other related results see the recent surveys \cite{Antolin11}
or \cite{CoornaertFS12}.

The result of Muller and Schupp was  stated in \cite{ms83} as a conjecture
and proved only  under the assumption that  finitely presented groups are accessible. 
The accessibility of finitely presented groups
was proved later by Dunwoody \cite{Dunwoody85}. (There are examples of 
finitely generated groups which not accessible by \cite{Dunwoody91}.) Accessibility means that the process of splitting the group with {S}tallings' structure theorem \cite{Stallings71} \footnote{ The structure theorem was first 
proved for finitely presented torsion-free groups by Stallings \cite{Stallings68} and for finitely generated torsion-free groups by Bergman \cite{Bergman68}.} eventually terminates.
In subsequent proofs the result in \cite{Dunwoody85} could be replaced by showing explicit upper bounds on how often splittings according to {S}tallings' structure theorem can be performed, see e.g. \cite{sen96dimacs}.

However, the reference to  \cite{Stallings71} remained.
Indeed, almost all proofs in the literature showing that a context-free group is virtually free use  the  structure theorem by {S}tallings. Recently, in \cite{Antolin11} another proof was given by Antolin which instead of {S}tallings' structure theorem and a separate result for accessibility uses a more general result due to Dunwoody~\cite{Dunwoody79}. 

The starting point for our contribution has been as follows: Circumvent the deep theorems of 
Dunwoody and Stallings by starting with a f.g.~group $G$ having a Cayley graph of finite treewidth.
Construct from these data a tree on which $G$ acts with finite node stabilizers and with finitely 
many orbits. Apply  Bass-Serre theory \cite{serre80} to see that $G$ can be realized as a 
fundamental group of a finite graph of groups with finite
vertex groups. It is known by \cite{Karrass73} that these groups are f.g.~and virtually free.

To follow this roadmap  became possible due to a recent paper by Kr{\"o}n \cite{kroen10} which presents a simplified version of Dunwoody's cut construction \cite{Dunwoody82}.
We realized that Kr{\"o}n's proof of {S}tallings' structure theorem can be modified such that it yields the tree we were looking for. We could not use Kr{\"o}n's result as a black box because in his paper he deals with cuts of globally minimal weight, only. Thus all cuts have the same weight whereas we need to consider cuts of different weight in order to get a non-refinable decomposition as fundamental group of a graph of groups.

Our approach leads to the following result: Let $\Gamma$ be a connected, locally finite graph of finite treewidth, and let $G$ be a group acting on $\Gamma$ such that $G\bs \Gamma$ is finite and each 
node stabilizer $G_v$ is finite. Then $G$ is finitely generated and virtually free. 

This is the essence of \prettyref{cor:nixstall}. To the best of our knowledge this result 
has not been formulated elsewhere. On the other hand, it is also clear that \prettyref{cor:nixstall} can be derived rather easily {}from 
existing results in the literature. So, the main contribution of the present paper is the new construction of optimally nested cuts (optimal cuts for short) and a direct self-contained
combinatorial proof of \prettyref{thm:new_alpha}, which implies \prettyref{cor:nixstall} by Bass-Serre theory.  

In \prettyref{thm:new2} we also give a new elegant self-contained proof for another important result in this area 
by Thomassen and Woess which is a consequence of \cite[Thm. {II} 2.20]{DicksD89}: Let $\Gamma$ be a locally finite, connected, accessible graph, and let a f.g.{} group $G$ act on $\Gamma$ such that $G\bs \Gamma$ is finite and each 
node stabilizer $G_v$ is finite. 
Then the group $G$ is accessible.

The outline of the paper is as follows: 
\prettyref{sec:prem} fixes some notation. 

In \prettyref{sec:cuts} we follow \cite{kroen10} introducing the necessary modifications. The focus in this section is on \emph{accessible graphs}
c.f.~\prettyref{dfn:accessible}. We work with bi-infinite simple paths rather  than with \emph{ends}.
This avoids some technical definitions and is more intuitive when drawing pictures as in 
\prettyref{fig:corners2}  or \prettyref{fig:corners3}. 
The key point in \prettyref{sec:cuts} is \prettyref{prop:opt_nested}, which is valid
for optimally nested cuts of different weights. It generalizes the corresponding results in \cite{Dunwoody82} and \cite{kroen10} on globally minimal cuts.
This leads to \prettyref{prop:tree} saying that the set of optimally nested cuts forms a tree set in the sense 
 of \cite{Dunwoody79}. This means that they can be viewed as the edge set of the so-called \emph{structure tree}.
 
In \prettyref{sec:blocks} we want to obtain some more information about the vertex stabilizers of the the action on the structure tree. In order to do so we define \emph{blocks} as in \cite{ThomassenW93}. The central result is \prettyref{prop:enden_in_Bc}. It says that blocks have at most one end, which finally leads to \prettyref{thm:new_alpha} and \prettyref{cor:new3}. 

\prettyref{sec:ftw} recalls the notion of \emph{finite treewidth}. The results 
of \prettyref{sec:cuts} and \prettyref{sec:blocks} yield the desired proof 
of \prettyref{thm:new_alpha}.

\prettyref{sec:ms} shows how to derive the result of Muller and Schupp \cite{ms83}
using our approach. This section does not contain any new material, but we tried to have a concise presentation. In particular, we omit the technical notion of a $k$-triangulation of a graph by showing 
directly that the Cayley graph of a context-free group has finite treewidth. 
This can be done with the very same ideas which are present in \cite{ms83}.
Then, we can apply \prettyref{cor:nixstall} to show that a context-free group is virtually free.

\section{Preliminaries}\label{sec:prem}

\subsection{Preliminaries on graphs}\label{sec:pgraph}
A \emph{directed graph} $\Gamma$ is given by the following data: 
A  set of vertices $V= V(\Gamma)$, a set of edges $E=E(\GG)$ together with two
mappings $s:E \to V$ and $t:E \to V$. The vertex $s(e)$ is the \emph{source} of $e$ and $t(e)$ is the \emph{target} of $e$.  A vertex $u$ and an edge $e$ are \emph{incident}, 
if $u \in \smallset{s(e), t(e)}$. The \emph{degree} of $u$ is the number of incident edges, and  $\Gamma$ is called \emph{locally finite} if the degree of all vertices is finite. 

An \emph{undirected graph} $\Gamma$ is a directed graph such that 
the set of edges $E$ is equipped with a fixed point free involution $e \mapsto \ov e$.
(i.e.,a map such that $e = \ov{\ov e} $ and $e \neq \ov e $ for all $ e \in E$).
Furthermore we  demand $s(e) = t(\ov e)$. 
An \emph{undirected edge} is the set
$\oneset{e, \ov e}$. By abuse of language we denote an undirected edge 
simply by $e$, too. 

If we speak about a \emph{graph}, then we always mean an undirected graph, otherwise 
we say specifically  directed graph. 
Most of the time we only consider (undirected) graphs without loops and multi-edges. In this case we identify $E$ with two-element sets of incident vertices
$\oneset{u, v}$ and write $e=uv$ if either  $s(e) =u$ and $t(e)=v$ or $s(\ov e) =u$ and $t(\ov e)=v$.

For $S\sse V(\GG)$ and $v \in V(\GG)$ define as usual in graph theory 
$\GG(S)$ (resp.{} $\GG - S$) to be the  subgraph of $\GG$ which is induced by  the vertex set $S$ (resp.{} $V(\GG) \sm S$)  
and $\GG - v = \GG - \{v\}$. We also write $\Comp S$ for the complement of $S$, i.e., $\Comp S = V(\GG) \sm S$.
Likewise for 
$e \in E(\GG)$ we let 
$\GG - e = (V(\GG), E(\GG) \sm \{e\})$.

A \emph{path} is a subgraph $(\{v_0,\ldots,v_n\}, \,\{e_1,\ldots,e_n\})$ such that $s(e_i) =v_{i-1}$ and $t(e_i)=v_{i}$  for all $1 \leq i \leq n$. It is \emph{simple} if 
the vertices are pairwise disjoint. It is \emph{closed} if $v_0=v_n$.
A \emph{cycle} is a closed path with $n \geq 3$ such that  $v_1,\ldots,v_n$ is a simple path. 

The \emph{distance} $d(u,v)$ between $u$ and $v$ is defined as the length (i.e., the number of edges) of the shortest path connecting $u$ and $v$. We let 
 $d(u,v) = \infty$ if there is no such path. 
A path  $v_0,\dots,v_n$ is called \emph{geodesic} if $n= d(v_0,v_n)$. An infinite path is defined as geodesic if all its finite subpaths are geodesic. For $A,B\sse V(\Gamma)$ the distance is defined as $d(A,B) = \min\set{d(u,v)}{u\in A,v\in B}$. 

A graph $\Gamma$ is called \emph{connected} if $d(u,v) < \infty$ for all vertices $u$ and $v$. 
A  \emph{tree}  is a connected graph without any cycle. If $T =(V,E)$ is a tree, we may fix a \emph{root} $r \in V$. This gives an orientation $E^+ \sse E$  by directing all edges ``away from the root''. 
In this way a rooted tree becomes a directed graph $(V, E^+)$ which refers to the 
tree $T= (V,E^+\cup E^-)$,  where  $E^- =  E\sm E^+$.

In the following, when we write $\Gamma$ we always mean a locally finite and connected graph, whereas the capital letter  $T$ refers to a tree, which does not need to be locally finite, in general.

\subsection{Preliminaries on groups}\label{sec:pgroup}
The paper is mainly concerned with finitely generated groups. Let $G$ be a group with $1$ as neutral element. The \emph{Cayley graph} 
$\Gamma$ of $G$ depends on $G$ and on a generating set  $X \sse G$.
It is defined by $V(\Gamma) = G$ and 
$E(\Gamma) = \set{(g,ga)}{g\in G \text { and } a\in X\cup X^{-1}}$, with the obvious incidence functions 
$s(g,ga) = g$, $t(g,ga) = ga$, and involution $\ov{(g,ga)} = (ga,a)$. 
For an edge $(g,ga)$ we call $a$ the label of $(g,ga)$ and extend this definition also to paths. Thus, the label of a path is a sequence 
(or \emph{word}) in the free monoid $X^*$.  The Cayley graph is without loops and without multi-edges. It is  connected because $X$ generates $G$. The Cayley graph 
$\Gamma$ is locally finite \IFF $X$ is finite. Sometimes we suppress $X$ if there is a standard choice for the generating set. For example, if
$G= F(X)$ is the free group over $X$, then the Cayley graph of $G$ refers 
to $X$ and it is a tree. By the \emph{infinite grid} we mean the Cayley graph of $\Z \times \Z$ with generators $(1,0)$ and $(0,1)$.

A group $G$ \emph{acts} on a graph  $\GG = (V,E)$ if there is an action of $G$ on $V$, denoted by $v\mapsto g\cdot v$,  and an action on $E$, denoted by $e\mapsto g\cdot e$, such that 
$s(g\cdot e) =  g\cdot s( e)$, $t(g\cdot e) =  g\cdot t( e)$, and $g\cdot\ov e = \ov{g\cdot e}$
for all $g \in G$ and $e \in E$. If $G$ acts on $\GG$, then we can define
its quotient graph $G \bs \GG$. Its vertices (resp.~edges) are the 
\emph{orbits} $G\cdot u$ for $u \in V$  (resp.~$G\cdot e$ for $e \in E$). 
We say that \emph{$G$ acts with finitely many orbits} if $G \bs \GG$ is finite. 

Let $\cG$ denote some class of groups. A group $G$ is called \emph{virtually $\cG$} if it has a subgroup of finite index which is in $\cG$.
Virtually finite groups are finite. The focus in this paper is 
on virtually free groups.

\section{Cuts and structure trees}\label{sec:cuts}
The constructions in this section follow the paper by Kr\"on \cite{kroen10} which gives a simplified approach to Dunwoody's constructions of cuts \cite{Dunwoody82}. The main difference between this section and the paper of Kr\"on lies in the definition of minimal cuts.
\subsection{Cuts and optimally nested cuts}
Let $\GG= (V(\GG),E(\GG))$ be a connected and locally finite graph. For a subset $C\sse V(\GG)$ we define the \emph{edge}- and \emph{vertex-boundaries} of $C$ as follows: 
\[
\begin{array}{lrcl}
\text{Edge-boundary: }&\delta C & =&\set{uv\in E(\Gamma)}{u\in C, v\in \Comp{C}}.\\
\text{Vertex-boundary: } &\beta C&=& \set{\; u\in V(\Gamma)}{\exists\, v\in V(\Gamma) \text{ with } uv\in \delta C}.
\end{array}
\]

\begin{dfn}\label{dfn:cut}
A \emph{cut} is a subset $C \subseteq V(\GG)$  such that the following conditions hold.
 \begin{enumerate}
\item $C$ and $\Comp{C}$ are non-empty and connected. 
\item $\delta C$ is finite.
\end{enumerate}
The \emph{weight} of a cut is defined by $\abs{ \del C}$. If $\abs{\delta C}\leq k$, then $C$ a called a \emph{$k$-cut}.
\end{dfn}
 
We are mainly interested in cuts where both parts $C$ and $\Comp{C}$ are infinite. 
However it might be that there are no such cuts. Consider the infinite grid
$\Z\times \Z$, i.e., the graph with vertex set $\Z\times \Z$ where $(i,j)$ is adjacent to the four vertices $(i,j\pm 1 )$ and $(i\pm 1,j)$. It is connected and locally finite, but there are no cuts of finite weight splitting the grid into two infinite parts.
 
%

The following well-known observation is crucial. It can be found e.g.{} in \cite{ThomassenW93} in a slightly different formulation:

\begin{lemma}
\label{lm:endl_k_cuts}
Let $S \sse V(\Gamma)$ be finite
and $k \geq 1$. There are only finitely many $k$-cuts $C$ with $\beta{C}\cap S \neq \es$.
\end{lemma}

\begin{proof}It is enough to prove the result for $S= \os{u,v}$ where $e=uv \in E(\GG)$ is some fixed edge. 
Since $\Gamma$ is locally finite, it is enough to show that the set of $k$-cuts $C$ with $e \in \delta C$ is finite. 
This is now trivial for $k=1$ because there is at most one cut with $\smallset{e} = \delta C$. If the graph $\Gamma -  e$ becomes disconnected,
i.e., $e$ is a so-called \emph{bridge}, then all cuts with $e \in \delta C$ have weight $k=1$. 
Thus, we may assume that  the graph $\Gamma -  e$ is still connected; and we may fix a path {}from $u$ to $v$ in  $\Gamma -  e$.
Every $k$-cut $C$ with $e \in \delta C$ becomes a $k-1$-cut $C$ in the graph  $\Gamma -  e$. Such a cut must use one edge of the 
path from $u$ to $v$ in  $\Gamma -  e$ because otherwise we had either both $u,v \in C$ or both 
$u,v \in \Comp C$. By induction, there are only finitely
many $k-1$-cuts using edges from a fixed path. Thus, we are done. 
\end{proof}

We are interested in bi-infinite simple paths which can be split into two 
infinite pieces by some  cut of finite weight. 
For a bi-infinite simple path $\alp$ denote:
\begin{align*}
\cC(\alp) &= \set{C \sse V(\GG)}{\text{$C$ is a cut and } \abs{\alp \cap C} = \infty = \abs{\alp \cap \Comp C}}.\\
\cC_{\min}(\alp) &= \set{C \in \cC(\alp)}{\text{$\abs{\del C}$ is minimal in }\cC(\alp)}.
\end{align*}

Thus, $\cC(\alp)\neq \es$ \IFF there is a cut of finite weight such that
the graph $\alp - \del C$ has exactly two infinite components each of these two being  a one-sided infinite subpath of $\alp$. 
We define the set of \emph{minimal cuts} $\cC_{\min}$ by 
 $$ \cC_{\min} = \bigcup\set{\cC_{\min}(\alp)}{\alp \text{ is a  bi-infinite simple path} }.$$
 In the infinite grid $\Z \times \Z$ we  have $\cC_{\min} = \es$.
Note that the set of minimal cuts may contain cuts of very different weight. 
Actually we might have $C,D \in \cC(\alp) \cap \cC_{\min}$
with $C \in \cC_{\min}(\alp)$, but  $D \notin \cC_{\min}(\alp)$. 
In such a case, there must be another  bi-infinite simple path $\bet$ 
with  $D \in \cC(\alp) \cap \cC_{\min}(\bet)$ and $\abs{\del C} < \abs{\del D}$. 
Here is an example: Let $\GG$ be the  subgraph of the infinite grid
$\Z \times \Z$ which is induced by the pairs $(i,j)$ satisfying  $j\in \{0,1\}$ or $i=0$ and $j\geq 0$. 
Let $\alp$ be the bi-infinite simple path with $i=0$ or $j=1$ and $i\geq 0 $ and let $\bet$ be the bi-infinite simple path defined by $j=0$. 
Then there are such cuts with $\abs{\del C} = 1$ and $\abs{\del D} = 2$, as depicted in \prettyref{fig:different_deltas}.

\begin{figure}[ht]
\begin{center}
\begin{tikzpicture}[scale = 0.7]
\def\width{5};
\def\height{4};

\draw (-\width -0.5,1) -- (\width +0.5,1) ;
\draw (-\width -0.5,0) -- (\width +0.5,0) ;
\draw (0,0) -- (0,\height + 0.5) ;
\node () at (-\width -1.1, 0.5) {\footnotesize{$\cdots$}};
\node () at (+\width +1.1, 0.5) {\footnotesize{$\cdots$}};
\node () at (0, \height + 1.3) {\footnotesize{$\vdots$}};
\node (oo) at (0,0){};
\node [below=2pt] {$(0,0)$};
\foreach \x in {1,...,\width}
{
	\draw (\x,1) -- (\x,0);
	\draw (-\x,1) -- (-\x,0);
}
\foreach \y in {2,...,\height}
{
	\draw (-0.075,\y) -- (0.075,\y);
}

\node () at (1.5, 2.0) {$\delta D$};
\draw[dashed] (1.5,1.5) -- (1.5,-0.5) ;
\node () at (-1.2, 2.5) {$\delta C$};
\draw[dashed] (-0.7,2.5) -- (0.7,2.5) ;

\draw[dotted,line width=1.4pt] (-\width -0.8,0) -- (\width +0.8,0) ;

\draw[dotted,line width=1.4pt] (0,1) -- (0,\height + 0.8) ;
\draw[dotted,line width=1.4pt] ( 0,1) -- (\width +0.8,1) ;
\node () at (\width +0.8, 1.4) {$\alpha$};
\node () at (\width +0.8, -0.4) {$\beta$};

\end{tikzpicture}
\end{center}
\caption[]{
The subgraph of the grid $\Z \times \Z$ induced by the pairs $(i,j)$ satisfying  $j\in \{0,1\}$ or $i=0$ and $j\geq 0$. Here we have $D \in \cC(\alp) \cap \cC_{\min}$ but $D \notin \cC_{\min}(\alp)$.}\label{fig:different_deltas}
\end{figure}
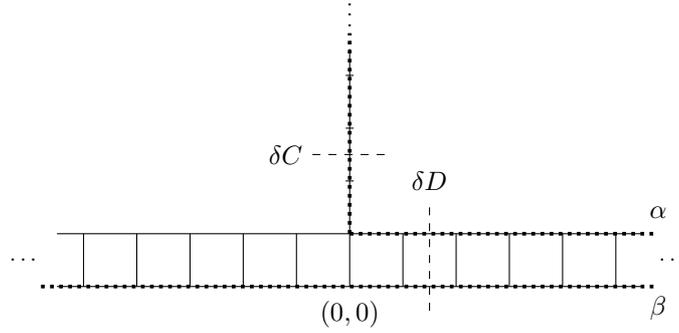

\begin{dfn}\label{dfn:M_m}
Two cuts $C$ and $D$ are called \emph{nested} if one of the four inclusions
$C\sse D$, $C\sse \Comp{D}$, $\Comp{C}\sse D$, or $\Comp{C}\sse\Comp{D}$ holds.
\end{dfn}

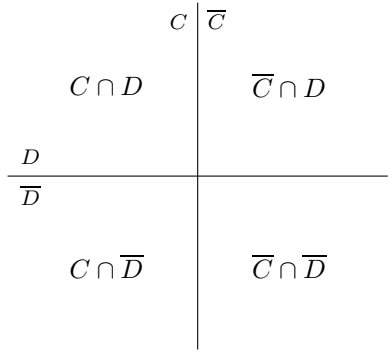
\begin{figure}[ht]
\begin{center}
\begin{tikzpicture}
\draw (0,2.3) -- (0,-2.3){};
\draw (2.5,0) -- (-2.5,0){};

\node (a) at (-1.2, 1.2) {$C\cap D$};
\node (a) at (1.2, 1.2) {$\Comp{C}\cap D$};
\node (a) at (-1.2, -1.2) {$C\cap \Comp{D}$};
\node (a) at (1.2, -1.2) {$\Comp{C}\cap \Comp{D}$};

\node[anchor=base] (a) at (-0.25, 1.95) {\footnotesize{$C$}};
\node[anchor=base]  (a) at (0.25, 1.95) {\footnotesize{$\Comp{C}$}};
\node (a) at (-2.2, 0.25) {\footnotesize{$D$}};
\node (a) at (-2.2, -0.25) {\footnotesize{$\Comp{D}$}};
\end{tikzpicture}
\end{center}
\caption[]{The corners of $C$ and $D$. Nested cuts have one empty corner.}\label{fig:corners1}
\end{figure}

The set $\oneset{C \cap D, C\cap \Comp{D}, \Comp{C}\cap D, \Comp{C}\cap \Comp{D}}$ is called the set of \emph{corners} of $C$ and $D$, see Figure~\ref{fig:corners1}. Two corners $E, E'$ of $C$ and $D$ are called \emph{opposite} if either $\smallset{E, E'} =\smallset{C \cap D, \, \Comp{C}\cap \Comp{D}} $ or $\smallset{E, E'} = \smallset{\Comp{C}\cap D, \, C\cap \Comp{D} }$. 
Two different corners are called \emph{adjacent} if they are not opposite.  
Note that two cuts $C,D$ are nested if and only if one of the four corners of $C$ and $D$ is empty.

 We define for every cut $C$ and $k \geq 1$ a cardinality $m_k(C)$ as follows: 
$$m_k(C) = \abs{\set{D}{\text{$C$ and $D$ are not nested and $D$ is a $k$-cut.}}}.$$ 

\begin{lemma}\label{lm:fred}
Let $k\in \N$ and $C$ be a cut, then $m_k(C)$ is finite.  
\end{lemma}

\begin{proof}
Let $S$ be a finite 
connected subgraph of $\GG$ containing all vertices of $\bet C$. 
The number of $k$-cuts $D$ with $\bet D \cap S \neq \es$
is finite by \prettyref{lm:endl_k_cuts}. For all other cuts we 
may assume (by symmetry) that $\bet C \sse D$.
Now assume that both, $C\cap \Comp D \neq \es$ and $\Comp C\cap \Comp D \neq \es$. Then we can connect a vertex $c \in C$ with some vertex $\ov c \in \Comp C$ inside the connected set  $\Comp D$. This must involve a vertex from 
$\bet C$, but $\bet C \sse D$. Hence, either $C\sse\Comp D$ or $\Comp C\sse\Comp D$. 

\end{proof}

We are mainly interested in graphs $\GG$ where the weight over all 
cuts in $\cC_{\min}$ can be bounded by some constant. This leads to the notion 
of \emph{accessible graph} due to \cite{ThomassenW93}: 
\begin{dfn}\label{dfn:accessible}
A graph is called \emph{accessible} if there exists a constant $k \in \N$ such that for every bi-infinite simple path $\alp$ either $ \cC(\alp) $ is empty or $ \cC(\alp) $ contains some $k$-cut
 \end{dfn}

For the rest of this section we  assume that $\GG$ is accessible. 
Thus, there is some constant $k$ such that  for all 
bi-infinite simple paths $\alp$ with $ \cC(\alp)\neq \es $ there exists some  cut $ C \in \cC(\alp) $ 
with $\abs{\delta C} \leq k$.

Fixing  this number $k$ let us define, by \prettyref{lm:fred},  for each cut $ C $  a natural number as follows: 
$$m(C) = \abs{\set{D }{\text{$C$ and $D$ are not nested and $D$ is a $k$-cut}}}.$$

We use the following notation, where $\alp$ denotes a bi-infinite simple path:
\begin{align*}
m_{\alp}&= \min \set{m(C)} {C\in \cC_{\min}(\alp)}\\
\Copt(\alp)&=  \set{C\in \cC_{\min}(\alp)} {m(C)=m_{\alp}}\\
\Copt&= \bigcup\set{\Copt(\alp)}{\alp \text{ is a  bi-infinite simple path} }
\end{align*}

\begin{dfn}\label{dfn:c_opt}
A cut $C\in \Copt$ is called an \emph{optimally nested cut}.
For simplicity, an optimally nested cut is also called \emph{optimal cut}.
\end{dfn}

In some sense we can forget all other cuts, we  just focus on 
optimal cuts. This viewpoint is possible because every ``cuttable'' bi-infinite simple path
is ``cut'' into two infinite parts at least by one optimal cut.
The next proposition is the main result in this section. 
\begin{proposition}\label{prop:opt_nested}
Let $C,D \in \Copt$. Then $C$ and $D$ are nested.
\end{proposition}

\begin{proof} We choose 
bi-infinite simple paths $\alp$ and $\bet$ such that 
$C\in \Copt(\alp)$ and $D \in \Copt(\bet)$. If possible, we let 
$\alp = \bet$.  In any case, we may assume that $m_{\alp}\geq m_{\bet}$.
The proof is by contradiction. Hence, 
we assume that $C$ and $D$ are not nested.

We distinguish two cases:
First, let  $D \in \cC_{\min}(\alp)$. Since $m(D) = m_{\bet}\leq m_\alp$, this implies 
$D \in \Copt(\alp)$ and therefore $\alp= \bet$. In particular, there are opposite corners $E$ and $E'$ such that $\abs{\alp \cap E} = \abs{\alp \cap E'}= \infty$. 

In the other case we have $D \notin \cC_{\min}(\alp)$
and therefore $\alp \neq \bet$. We claim that there must be one corner $K$ of $C$ and $D$ such that $\abs{\alp \cap K} <  \infty$ and $\abs{\bet \cap K} <  \infty$. Indeed, if there is no such corner $K$, then infinite parts of $\alp$ and $\bet$ are in opposite corners. 
  In particular, $\alpha$ and $\beta$ are split by both by  $C$ as well as by  $D$ in two infinite pieces. This implies $\abs{\delta C}= \abs{\delta D}$, and hence $D \in \cC_{\min}(\alp)$. Thus such a corner $K$ exists and we define
 $E$ and $E'$ to be the adjacent corners of $K$. Without loss of generality, $E$ splits $\alpha$ into two infinite pieces and $E'$ splits $\beta$ into two infinite pieces.

\begin{figure}[ht]
\begin{center}
\begin{tikzpicture}
\draw (0,2.3) -- (0,-2.3){};
\draw (2.5,0) -- (-2.5,0){};
\draw[snake=snake, segment length = 9mm] (1.6, -2.) --(-1.6, 1.6) ;
\node() at (-1.35,1.6){$\alpha$} ;
\draw[snake=snake, segment length = 9mm] (1.5, 1.6) --(-1.,- 2);
\node() at (1.75,1.6) {$\beta$};

\node[anchor=base] (a) at (-0.25, 1.95) {\footnotesize{$C$}};
\node[anchor=base]  (a) at (0.25, 1.95) {\footnotesize{$\Comp{C}$}};
\node (a) at (-2.2, 0.25) {\footnotesize{$D$}};
\node (a) at (-2.2, -0.25) {\footnotesize{$\Comp{D}$}};
\end{tikzpicture}
\end{center}
\caption[]{For all four corners $K$ we have $\max\{\abs {K \cap \alp}, \abs {K \cap \bet}\} = \infty$.}\label{fig:corners2}
\end{figure}
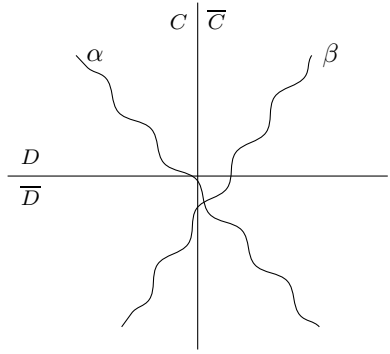

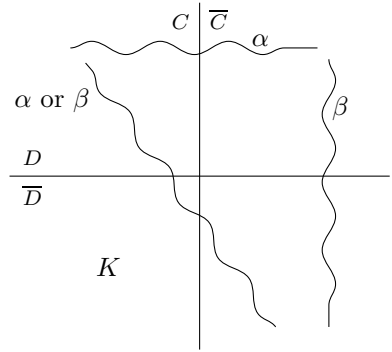
\begin{figure}[ht]
\begin{center}
\begin{tikzpicture}
\draw (0,2.3) -- (0,-2.3){};
\draw (2.5,0) -- (-2.5,0){};

\draw[snake=snake, segment length = 9mm] (1., -2) --(-1.5, 1.5) ;
\node() at (-1.94,1.0){$\alpha$ or $\beta$} ;
\draw[snake=snake, segment length = 9mm] (1.7, 1.55) --(1.7,- 2);
\node() at (1.85,0.9) {$\beta$};
\draw[snake=snake, segment length = 9mm] (-1.7, 1.7) --(1.55, 1.7);
\node() at (0.8,1.8) {$\alpha$};

\node (a) at (-1.2, -1.2) {$K$};

\node[anchor=base] (a) at (-0.25, 1.95) {\footnotesize{$C$}};
\node[anchor=base]  (a) at (0.25, 1.95) {\footnotesize{$\Comp{C}$}};
\node (a) at (-2.2, 0.25) {\footnotesize{$D$}};
\node (a) at (-2.2, -0.25) {\footnotesize{$\Comp{D}$}};
\end{tikzpicture}
\end{center}
\caption[]{For one corner $K$ we have 
$\max\{\abs {K \cap \alp}, \abs {K \cap \bet}\} <\infty$.}
\label{fig:corners3}
\end{figure}

Thus, in both cases, $E$ and $E'$ are defined such that 
$\abs{\alp \cap E} = \abs{\bet \cap E'}= \infty$.
  By interchanging, if necessary,  $C$ with $\Comp  C$ and 
  $D$ with $\Comp  D$,  we may assume  that 
  $E= C \cap D$ and $E' = \Comp C \cap \Comp  D$, too. 
  
  Thus, in all cases we are in the following situation: 
  
  $C$ and $D$ are not nested, $C\in \Copt(\alp)$, 
  $D\in \Copt(\bet)$, 
  $E= C \cap D$, $E' = \Comp C \cap \Comp  D$, and $\abs{\alp \cap E} = \abs{\bet \cap E'}= \infty$. Possibly $\alp =\bet$, 
but it is not yet clear that $E$ and $E'$ are cuts. 

The graph $\GG(E)$ contains an infinite connected component $F \sse E$ 
such that $\abs{\alp \cap F} = \infty$. Let us show that 
$\Comp F$ is non-empty and  connected. The set $\Comp F$ is non-empty
and infinite because 
$E ' \sse \Comp F$. Now fix a vertex $v\in E'$ and let 
$u\in \Comp F$. There is a path $\gam$ from $u$ to $v$ in $\GG$ and on this path there is a first vertex $w$ with $w \in \Comp C \cup \Comp  D$.
If the initial path from $u$ to $w$ was using a point of $F$, then it would 
be a path in $E$, and $u$ would be in the connected component $F$, which was excluded. Hence, we can connect $u$ to $w$ in $\GG -F$. Now, by symmetry 
 $w \in \Comp C$. But then $w, v \in \Comp C \sse \GG -F$ and $\Comp C$ is connected. Hence, $F$ is a cut. 
 
 In a symmetric way we find a cut $F' \sse E'$ such that $\abs{\bet \cap F'} = \infty$.
 Let us  show that $F = E \in  \cC_{\min}(\alp)$ and $F' = E'\in  \cC_{\min}(\bet)$.
 
 We can write $\abs{\del E} = a + b +c +d$, where
 \begin{align*}
a = & \abs{\set{xy}{x \in F \wedge y \in \Comp C \cap D}} \\
b = & \abs{\set{xy}{x \in F \wedge y \in E'}} \\
c = & \abs{\set{xy}{x \in F \wedge y \in C \cap \Comp D}} \\
d = & \abs{\set{xy}{x \in E\sm F \wedge y \notin E}} 
\end{align*}
Likewise, we have $\abs{\del E'} = a' + b' +c' +d'$, where
 \begin{align*}
a' = & \abs{\set{xy}{x \in F' \wedge y \in \Comp C \cap D}} \\
b' = & \abs{\set{xy}{x \in F' \wedge y \in E}} \\
c' = & \abs{\set{xy}{x \in F' \wedge y \in C \cap \Comp D}} \\
d' = & \abs{\set{xy}{x \in E'\sm F' \wedge y \notin E'}} 
\end{align*}
With the minimality of $\abs{\delta C}$ and $\abs{\delta D}$ we derive the following: 
\begin{align*}
a+b+c' \leq  &  \abs{\del C} \leq \abs{\del F}  =  a+b+c \\
  a'+b'+c  \leq  & \abs{\del D}  \leq  \abs{\del F'}  =  a'+b'+c'
\end{align*}
We conclude $\abs{\del C} = \abs{\del F}$ and $\abs{\del D}  = \abs{\del F'}$. 
This implies $F  \in  \cC_{\min}(\alp)$ and $F'\in  \cC_{\min}(\bet)$. 

We still have to show $E = F $ and $E'= F'$. For this it is enough to show that
$d = d' =0$. Assume by contradiction that $d + d' \geq 1$. Say, $d\geq 1$. Then we 
have $\abs{\del C} + \abs{\del D} > a+b+c+a'+b'+c'$. 
 This  contradicts the assertion 
$\abs{\del C} = \abs{\del F}$ and $\abs{\del D}  = \abs{\del F'}$. 
This yields $F= E  \in  \cC_{\min}(\alp)$ and $F'= E'\in  \cC_{\min}(\bet)$.
Since $C$ and $D$ are optimal cuts, we conclude 
$m(E) \geq m(C)$ and $m(E') \geq m(D)$.

The crucial step in the proof is the following assertion:
\begin{equation}\label{eq:frodo}
m(E)+m(E')<m(C)+m(D).
\end{equation}
Once we have established \prettyref{eq:frodo} we get an obvious contradiction 
to $m(E) \geq m(C)$ and $m(E') \geq m(D)$.

To see \prettyref{eq:frodo}, we show two claims: 
\begin{enumerate}
\item If a cut $F$ is  nested with $C$ \emph{or}  nested with  $D$, then $F$ is  nested with $E$ \emph{or}  nested  with $E'$:\label{eq_cap_nested}

By symmetry let $F$ be nested with $C$. 
If 
$F\sse C$ (resp.{} $\Comp F\sse C$), then $F\sse \Comp {E'}$ (resp.{} $\Comp F \sse \Comp {E'}$). 
If 
$C\sse F$ (resp.{} $C \sse \Comp F $), then $E \sse F$ (resp.{} $E \sse \Comp {F}$).

\item  If a cut $F$ is nested with $C$ \emph{and} nested with $D$, then $F$ is nested with $E$ \emph{and} nested with $E'$:\label{eq_cup_nested}

By symmetry in $F,\Comp{F}$ we may assume $C \sse F$ or $\Comp C \sse F$.
Using now the symmetry in $E,E'$
 we may assume that $C\sse F$. Hence we have $E\sse F$;  and it remains to show
 that $E'$ and $F$ are nested. 
 For $D \sse \Comp F$, we had $C \cap D = \es$. 
 For $\Comp D \sse \Comp F$, we had $C \cap \Comp D = \es$.
 Both is impossible because $C$ and $D$ are not nested. 
 For $D \sse F$ we obtain $\Comp{E'} =  C \cup D \sse F$ what implies that $E'$ and $F$ are nested. 
 Finally let $\Comp D \sse F$, then $E' \sse F$. Again $E'$ and $F$ are nested.
\end{enumerate}
\renewcommand{\labelenumi}{\arabic{enumi}.}

Putting claims \ref{eq_cap_nested} and \ref{eq_cup_nested} together yields: 
$ m(E)+m(E')\leq m(C)+m(D).$
Now, $C$ is nested with both corners $E$ and $E'$. Hence, $C$  is not counted on the left-hand side 
of the inequality. However, $C$ is  counted on the right-hand side because $C$  is not nested with $D$. That means the inequality in \prettyref{eq:frodo} is strict. Hence, we have shown the result of the proposition. 
\end{proof}

Analog results to \prettyref{prop:opt_nested} are Theorem 1.1 in \cite{Dunwoody82} or Theorem 3.3 of \cite{kroen10}. 
In contrast to these results, \prettyref{prop:opt_nested} allows that $\Copt$ may contain cuts of different weights. 
We have to deal with cuts of different weights because we wish to get a ``complete'' decomposition of virtually free groups like 
$(\Z \times \Z/2\Z)*\Z/2\Z$. Like in the graph in \prettyref{fig:different_deltas}, in the Cayley graph of this group cuts with weight $1$ and $2$ are necessary to split all bi-infinite paths into two infinite pieces, see \prettyref{fig:Cayley_graph_ZZ2Z2}. 

\begin{figure}[ht]
\begin{center}
\ifdraft Where have all  my UFOs gone? Time is like a jet plane. It moves too fast.
\else
\input{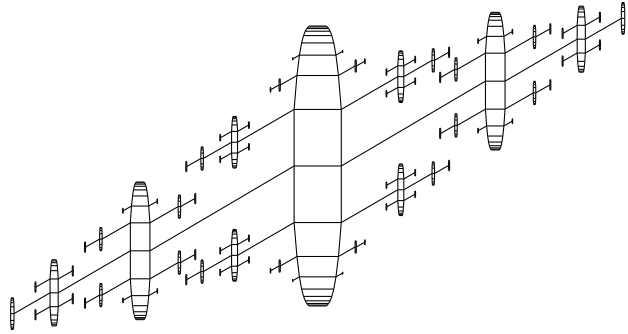}
\fi
\end{center}
\caption[]{The Cayley graph of the group $(\Z \times \Z/2\Z)*\Z/2\Z$
moving very fast towards your eyes.}\label{fig:Cayley_graph_ZZ2Z2}
\end{figure}

\subsection{The structure tree}\label{vdsec:tree_set} 
The notion of \emph{structure tree} is due to Dunwoody \cite{Dunwoody79}.
Recall that $\GG$ is assumed to be accessible, hence  $\Copt$ is defined and there is some $k \in \N$ such that every cut in $\Copt$ is a $k$-cut. 

\begin{lemma}\label{lm:tree_set}
Let $C,D\in\Copt$. Then the set $\set{E \in \Copt}{C\sse E\sse D}$ is finite. 
\end{lemma}

\begin{proof}
Choose two vertices $u\in C$ and $v \in \Comp{D}$, and a path 
$\gam$ in $\GG$ connecting them. Every cut $E$ with $C\sse E\sse D$ must separate $u$ and $v$ and thus contain a vertex of of $\gam$. With \prettyref{lm:endl_k_cuts} and the accessibility of $\Gamma$ it follows that there are only finitely many such cuts.
\end{proof}

The set $\Copt$
is partially ordered by $\sse$. By \prettyref{lm:tree_set}, the partial order is induced by 
its so-called \emph{Hasse diagram}. In the Hasse diagram there is an arc from 
$\Comp C\in \Copt$ to  $D \in \Copt$ \IFF $\Comp C \ssnq D$ and there is no $E\in \Copt$ between them. In dense orderings, like $(\Q, \leq)$, the Hasse diagram is empty, whereas in discrete orderings, like $(\Copt,\sse)$, the partial order is the reflexive and  transitive closure of the arc relation in the Hasse diagram.

If there is an arc from $\Comp C$ to  $D$, then there is also an  arc from $\Comp D$ to $ C$. In such a situation we put $C$ and $D$ 
in one class: 
\begin{dfn}\label{dfn:cuts_equiv}
For $C$, $D \in \Copt$ we define the relation $C\sim D$ by the following condition:  
\begin{equation*}
\text{Either  } C=D \text{ or both }  \Comp{C} \subsetneqq D \text{ and } \forall\; E \in \Copt:  \,\Comp{C} \subsetneqq E \subseteq D \implies D =E.
\end{equation*}
\end{dfn}
The intuition behind this definition is as follows: Consider $(C, \Comp C)$ for  $ C \in \Copt$ as an
edge set of some graph. Call edges $(C, \Comp C)$ and $(D, \Comp D)$
to be adjacent if $C \sim D$. This makes sense due the following property. 

\begin{lemma}\label{lm:frida}
The relation $\sim$ is an equivalence relation.
\end{lemma}
\begin{proof}
Reflexivity and symmetry are immediate. Transitivity requires to  check all inclusions how the cuts can be nested. A proof can be found e.g.{} in \cite{Dunwoody79}. In order to keep the paper self-contained we repeat the  proof for transitivity. Let  $C\sim D\neq C$ and  $D\sim E \neq D$.
This implies $\es \neq \Comp D \sse C\cap E$. 
We have to show that $C\sim E$. The  cuts $C$ and $E$ are nested due to 
\prettyref{prop:opt_nested}. Hence we have one of the following four
inclusions: 
\begin{itemize}
\item $C\sse E$: This implies $\Comp{D} \ssnq C \sse E$. Hence, $C=E$ because $D\sim E$.
\item $E\sse C$: This implies  $\Comp{D} \ssnq E \sse C$, Hence, $C=E$ because $D\sim C$.
\item $E\sse \Comp{C}$: This contradicts  $C\cap E\neq \emptyset$.
\item $\Comp{C} \sse E$: Since  $C\cap E\neq \emptyset$, we see $\Comp{C} \ssnq E$. Now, let  $\Comp{C} \ssnq F\sse E$ for some $F \in \Copt$. 
Since $F$ and D are nested, we obtain one of the following inclusions:
\begin{itemize}
\item $D\sse F$: This implies $D \sse  E$, in contradiction to  $\Comp{D} \ssnq E$. 
\item $F\ssnq D$: This implies  $\Comp{C} \ssnq F \ssnq D$, in contradiction to  $C\sim D$. 
\item $F\sse \Comp{D}$: This implies  $\Comp{C} \ssnq F \sse \Comp{D}$, in contradiction to   $\Comp{C} \ssnq D$.
\item $\Comp{D} \ssnq F$: This implies  $\Comp{D} \ssnq F \sse E$. Hence, $F=E$ because $D\sim E$.
\end{itemize}
\end{itemize}
\end{proof}

\begin{dfn}
Let  $T(\Copt)$ denote the following graph:
\begin{align*}
V(T(\Copt))&=\set{[C]}{C \in \Copt}\\
E(T(\Copt))&=\set{(C,\Comp{C})}{C \in \Copt}
\end{align*}
The  incidence maps are defined by $s((C,\Comp{C}))=[C]$ and $t((C,\Comp{C}))=[\Comp{C}]$.
\end{dfn}
The directed edges are in canonical bijection with the pairs $([C],[\Comp{C}])$. Indeed, let  $C\sim D$ and $\Comp{C}\sim \Comp{D}$. It follows $C=D$ because otherwise $C\ssnq \Comp{D}\ssnq C$.
Thus, $T(\Copt)$ is an undirected graph without self-loops and multi-edges. 

The graph  $T(\Copt)$ is locally finite \IFF the equivalence classes $[C]$ are finite. Hence it is not locally finite, in general: For each $C \in \Copt$  there might be infinitely many $D \in \Copt$ with $D \sim C$. For example, consider the Cayley graph of the group $(\Z\times \Z)  *\Z/2\Z$. There is one $\Z\times \Z$-plane through the origin and for every $(i,j) \in \Z\times \Z$ there is one edge leaving this plane. Removing this edge  defines a 
 unique $1$-cut  $C_{i,j}$ with $ \Z\times \Z\sse C_{i,j}$. It is in $\Copt$ because all other minimal cuts are nested 
with $C_{i,j}$. We have  $\Comp{C_{i,j}} \sse C_{0,0}$ for $(i,j) \neq (0,0)$, but there is no $E \in \Copt$  with $\Comp{C_{i,j}} \subsetneqq E \subsetneqq C_{0,0}$.

\begin{proposition}\label{prop:tree}
The graph $T(\Copt)$ is a tree.
\end{proposition}

\begin{proof}
Let $\gamma$ be a simple path in $T(\Copt)$ of length at least two.
Then $\gamma$ corresponds to a sequence of cuts
\[C_0,\, \Comp{C}_0 \sim C_1, \dots,  \Comp{C}_{n-2} \sim C_{n-1},\,\Comp{C}_{n-1}= C_n\]
with $[C_{i-1}] \neq [C_{i+1}]$, so in particular $\Comp{C}_{i-1} \neq C_i$ for $1\leq i\leq n-1$ (otherwise we would have $C_{i-1}= \Comp{C}_i \sim  C_{i+1}$).
So we get a sequence
\[C_0\subsetneqq C_1\subsetneqq C_2\ssnq\, \cdots  \, \subsetneqq C_{n-1}\]
Therefore we have $C_0 \neq \Comp{C}_{n-1}$ and $\Comp{C}_0 \not\sse\Comp{C}_{n-1}$. So $C_0 \not\sim \Comp{C}_{n-1}=C_n$ and the original path is not a cycle. So $T(\Copt)$ has no cycles.

It remains to show that $T(\Copt)$ is connected.
Let $[C],[D] \in V(T(\Copt))$. Since $C$ and $D$ are nested and $(C,\Comp{C}),(D,\Comp{D})\in E(T(\Copt))$, we may assume $C\sse D$. 
By \prettyref{lm:tree_set}, there are only finitely many cuts $E \in \Copt$, with $C\sse E\sse D$. Now, let $C_0, C_1,\dots,  C_n$ be a not refinable sequence of cuts in $\Copt$ such that
\[C=C_0\subsetneqq C_1\subsetneqq C_2 \ssnq\, \cdots  \, \subsetneqq C_{n-1} \subsetneqq C_n=D\]
Then we obtain a path from $C$ to $D$:
\[C=C_0,\, \Comp{C}_0 \sim C_1,\,  \Comp{C}_1 \sim C_2,\dots,   \Comp{C}_{n-1} \sim C_n=D\]
Hence, $T(\Copt)$ is connected and therefore a tree.
\end{proof}

\begin{remark}\label{rem:dunwoody}
 According to Dunwoody \cite{Dunwoody79} a \emph{tree set} is 
a set  of pairwise nested cuts, which is closed under complementation and 
such that for each $C,D\in\cC$ the set $\set{E \in \cC}{C\sse E\sse D}$ is finite. Thus, using this terminology, \prettyref{prop:opt_nested} and \prettyref{lm:tree_set} show that
$\Copt$ is a tree set. Once this is established \prettyref{prop:tree} becomes a general fact due to Dunwoody \cite[Thm.~2.1]{Dunwoody79}.
\end{remark}

\section{Actions on $\GG$ and its structure tree $T(\Copt)$}\label{sec:blocks}
In this section, $\GG$ denotes a connected, locally finite, and accessible graph such that the group of automorphisms $\Aut(\GG)$ acts with finitely many  orbits on $\GG$. The action on $\GG$ induces an action of $\Aut(\GG)$ on $\Copt$ and on the structure tree $T(\Copt)$. 
For example, if $\GG$ is the Cayley graph of a group $G$ with respect to some finite generating set $\Sg \sse G$, then $\GG$ is connected, locally finite, 
and there is only one orbit: $\abs{\Aut(\GG)\bs \GG} = 1$. 

\begin{lemma}\label{lm:VDendl_orbits_k_cuts}
Let $\abs{\Aut(\GG)\bs \GG}$ be finite and 
$k\in \N$. Then the  canonical action of $\Aut(\GG)$ on the set of $k$-cuts has 
finitely many orbits, only. In particular $\Aut(\GG)$  acts on 
$\Copt$ and on the tree $T(\Copt)$ with finitely many orbits.
\end{lemma}

\begin{proof} Let $\Aut(\GG)\bs V(\Gamma)$ be represented by some finite vertex set $U\sse V(\GG)$.
With \prettyref{lm:endl_k_cuts} it follows that there are only finitely many $k$-cuts $C$ such that $U\cap \beta C \neq \es$. Since every cut is in the same orbit as some cut $C$ with $U\cap \beta C \neq \es$, the  group $\Aut(\GG)$ acts on the set of $k$-cuts with 
finitely many orbits.

Since $\Gamma$ is accessible, there is a $k$ such that for all cuts $C \in \Copt$ holds $\abs{\delta C} \leq k$. 
For the last statement observe that $\set{(C, \Comp C)}{C\in \Copt}$ is the edge set of $T(\Copt)$. Thus,  
 the action of $\Aut(\GG)$ on  $T(\Copt)$ has only finitely many orbits, too. 
\end{proof}

 For $S\sse V(\Gamma)$ and $k\geq 1$ we let $N^kS=\set{v \in V(\Gamma)}{d(v,S)\leq k}$ denote the $k$-th neighborhood of $S$. 
Now, for a cut $C$ we can choose $k$ large enough such that $N^k C\cap \Comp{C}$ is connected because $\Comp{C}$ is connected. (Indeed, all points in $\beta C \cap \Comp{C}$ can be connected 
 in $\Comp{C}$, hence for some $k$ large enough these points can be connected in  $N^k C\cap \Comp{C}$. This $k$ suffices to make $N^k C\cap \Comp{C}$
 connected.) 
By \prettyref{lm:VDendl_orbits_k_cuts}, there are only finitely many orbits of optimal cuts.  
 Thus we can choose some 
 $\kappa\in\N$ which works for all $C\in \Copt$ 
 and fix it for the rest of this section. 
 
Now, we want to deduce some more information about the structure of the vertex stabilizers $G_{[C]}= \set{g \in G}{gC\sim C}$ of vertices of the tree $T(\Copt)$. Therefore, we assign to each vertex of $T(\Copt)$ a so-called \emph{block}. The definition has been taken from \cite{ThomassenW93}. In \prettyref{lm:endl_orbits_Bc} we show that the blocks are somehow ``small''. They are defined as follows.
\begin{dfn}\label{dfn:B{[C]}}
Let $\Aut(\GG)\bs \GG$ be finite and $\Copt$ be the set of optimal cuts.
 Let $\kappa\geq 1$ be defined as above such that $N^\kappa C\cap \Comp{C}$ is connected for all $C\in \Copt$. The \emph{block} assigned to $[C] \in V (T(\Copt))$ is defined by:
\[B{[C]}= \bigcap_{D\sim C}  N^\kappa D\]
\end{dfn}

\begin{lemma}\label{lm:VconBc}
We have   $$B{[C]} = \bigcap_{D \sim C } D \; \cup \; 
\bigcup_{D \sim C }N^\kappa D \cap \Comp{D}.$$
\end{lemma}

\begin{proof}
The inclusion from left to right is trivial. 
 It is therefore enough to show that we have $N^\kappa C \cap \Comp{C}\sse B{[C]}$.
 Clearly, $N^\kappa C \cap \Comp{C}\sse N^\kappa C$. Thus is enough to consider 
 $D\sim C$, $D\neq C$ and to show that $N^\kappa C \cap \Comp{C}\sse N^\kappa D$.
 This follows from: 
 \[N^\kappa C \cap \Comp{C} \sse  \Comp{C} \ssnq D \sse N^\kappa D\]
\end{proof}

\begin{example}
\prettyref{fig:schnitte_in_cayleygraph} shows a part of the Cayley graph of the free product $\Z/2\Z\,*\,\Z/3\Z = \genr{a,b}{a^2=1=b^3}$. The minimal cuts cut the edges with label $a$,  i.e., they cut through cosets of $\Z/2\Z$. The optimal cuts are exactly the minimal cuts. The three cuts depicted with dashed lines belong to the same equivalence class and the bold vertices form the respective block. Here, we can choose $\kappa=1$ for the definition of the blocks.
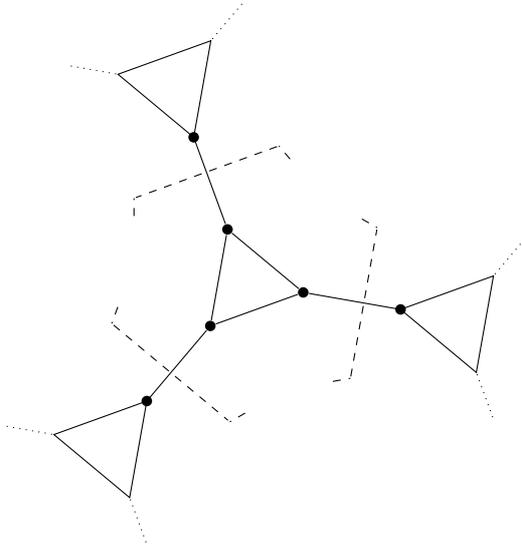
\begin{figure}[ht]
\begin{center}
\begin{tikzpicture}[scale=.65, rotate=20]
\tikzstyle{every node}=[inner sep=0pt]

\node[circle, fill,inner sep=1.4pt] (1) at (0, 0) {};
\node[circle, fill,inner sep=1.4pt] (a) at (2, 0) {};
\node[circle, fill,inner sep=1.4pt] (aa) at (1,1.732050807568877294 ) {};

\node (aaba) at (2, 5.464101615137754588) {};
\node (aabaa) at (0,5.464101615137754588) {};
\draw[dotted] (aaba)-- ++(30:1);
\draw[dotted] (aabaa)-- ++(150:1);
\path (aaba) -- ++(30:2) node (aabac){};
\path (aabaa) -- ++(150:2) node (aabaac){};

\node (abaa) at (5.732050807568877294, -1) {};
\node (aba) at (4.732050807568877294, -2.732050807568877294) {};
\draw[dotted] (aba)-- ++(270:1);
\draw[dotted] (abaa)-- ++(30:1);
\path (aba) -- ++(270:2) node (abac){};
\path (abaa) -- ++(30:2) node (abaac){};

\node (ba) at (-3.732050807568877294, -1) {};
\node (baa) at (-2.732050807568877294, -2.732050807568877294) {};
\draw[dotted] (ba)-- ++(150:1);
\draw[dotted] (baa)-- ++(270:1);
\path (ba) -- ++(150:2) node (bac){};
\path (baa) -- ++(270:2) node (baac){};

\draw (aa)-- ++(90:2) node[circle, fill,inner sep=1.4pt] {}  -- ++(60:2) -- ++(180:2) -- ++(300:2){};
\draw (a)-- ++(330:2) node[circle, fill,inner sep=1.4pt] {} -- ++(300:2)-- ++(60:2) -- ++(180:2) {};
\draw (1)-- ++(210:2) node[circle, fill,inner sep=1.4pt] {} -- ++(180:2)-- ++(300:2)-- ++(60:2)  {};

\path (baac)  -- ++(45:4) node (baacd){};
\draw [style=dashed] (baacd)  -- ++(190:0.4) node (baacde){};
\path (bac)  -- ++(15:4) node (bacd){};
\draw [style=dashed] (bacd)  -- ++(230:0.4) node (bacde){};
\draw [style=dashed] (bacde)  -- (baacde){};

\path (abaac)  -- ++(165:4) node (abaacd){};
\draw [style=dashed] (abaacd)  -- ++(310:0.4) node (abaacde){};
\path (abac)  -- ++(135:4) node (abacd){};
\draw [style=dashed] (abacd)  -- ++(350:0.4) node (abacde){};
\draw [style=dashed] (abacde)  -- (abaacde){};

\path (aabaac)  -- ++(285:4) node (aabaacd){};
\draw [style=dashed] (aabaacd)  -- ++(70:0.4) node (aabaacde){};
\path (aabac)  -- ++(255:4) node (aabacd){};
\draw [style=dashed] (aabacd)  -- ++(110:0.4) node (aabacde){};
\draw [style=dashed] (aabacde)  -- (aabaacde){};

\draw (1) -- (a){};
\draw (1) -- (aa){};
\draw (aa) -- (a){};

\end{tikzpicture}
\end{center}
\caption[]{Block of six vertices in the Cayley graph of $\Z/2\Z*\Z/3\Z$}\label{fig:schnitte_in_cayleygraph}
\end{figure}
\end{example}

\begin{lemma}\label{lm:VDconBc}The following assertions hold.
\begin{enumerate}
\item\label{conBci} For all $C \in \Copt$ the block $B{[C]}$ is connected .
\item\label{conBcii} There is a number $\ell \in \N$ such that for all $C \in \Copt$ 
and all $S \sse B{[C]}$ we have: 
Whenever two vertices $u,v \in B{[C]}- N^\ell S$ can be connected by some path in $\GG- N^\ell S$, then they can be connected by some path in $B{[C]}-  S$.
\end{enumerate}
\end{lemma}

\begin{proof}
Note that \ref{conBci}.~is a special case of \ref{conBcii}.~by choosing 
$S = \es$. Let $\ell = \max\set{d(u,v)}{D\in \Copt, \, u,v \in \Comp D\cap N^\kappa D}$. Thus $\ell$ is a uniform bound on the diameters for the sets $N^\kappa D \cap \Comp{D}$ for $D\in \Copt$. It exists because there are only finitely many orbits of optimal cuts.

Now, let  $u, v \in B{[C]}-N^\ell S$ be two vertices
which are connected by some path $\gam$ in $\GG- N^\ell S$. 
We are going to transform the path $\gam$ into some path $\gam'$
where all vertices are in $B{[C]}-  S$. If $\gam$ is entirely in $B{[C]}$ we are done. Hence we may assume that there exist
a first vertex $v_m$ of $\gamma$ which does not lie in $B{[C]}$.
Thus for some $D \sim C$ we have $v_m \notin N^\kappa D$. Since
$\kappa \geq 1$, we have $v_{m-1} \in N^\kappa D \cap \Comp{D}$. 
For some $n>m$ we find a vertex 
$v_{n}$ which is the first vertex after $v_m$ lying in $N^\kappa D$ again.
As $v_{n}$ is the first one, we have $v_{n} \in N^\kappa D \cap \Comp{D}$, too. 
 Since $N^\kappa D \cap \Comp{D}$ is connected, we can choose a path from $v_{m-1}$ to $v_{n}$ inside $N^\kappa D \cap \Comp{D}$. This is a path  inside $B{[C]}$ by \prettyref{lm:VconBc}. Note that this path does not use $v_m$ anymore. Moreover, the new segment cannot meet any point in $S$ because otherwise $v_n \in N^\ell S$. The path from $v_n \in B{[C]}- N^\ell S$ to $v$ is shorter than $\gam$. Hence, by induction, $v_n$ is connected to $v$ in $B{[C]}-  S$; and we can transform $\gam$ as desired. 
  \end{proof}

\begin{lemma}\label{lm:blocks}
Let $C \in \Copt$ and $ g \in \Aut(\GG)$ be such that $g(C) \sim C$. 
Then we have  $g(B{[C]}) = B{[C]}$.
\end{lemma}

\begin{proof}
Let $v \in B{[C]} = \bigcap \set{N^\kappa D}{D\sim C}$, then
$g v \in \bigcap\set{N^\kappa gD }{gD\sim gC}$ for all $ g \in \Aut(\GG)$.
Now, 
 if $D\sim C$ and $gC \sim C$ for some $g\in G$, then $gD \sim gC \sim C$. Hence, $gv \in B{[C]}$.
\end{proof}

\begin{lemma}\label{lm:endl_orbits_Bc}
Let $\GG$ be a connected, locally finite, and accessible graph such that a group $G$ acts on $\Gamma $ with finitely many orbits. Let $C\in \Copt$. 
Then the stabilizer $G_{[C]}= \set{g \in G}{gC\sim C}$ of the vertex 
$[C]= \set{D}{C \sim D} \in V(T(\Copt))$
acts with finitely many orbits on the block $B{[C]}$.
\end{lemma}

\begin{proof}
Since $G$ acts with finitely many orbits on $\Gamma$, it acts with finitely many orbits on the set $\Copt$. For $D \sim gD \sim C$ we have $g \in G_{[C]}$ by \prettyref{lm:blocks}. Hence, 
$G_{[C]}$ acts with finitely many orbits on $[C]$. This implies that
$G_{[C]}$ acts with finitely many orbits on the union $\bigcup\set{\beta D}{D \sim C}$.

We are going to show that there is some $m\in \N$ such that for every $v \in B{[C]}$  there is a cut $D\in [C]$ with $d(v,\beta D)\leq m$. This implies the lemma since $\Gamma$ is locally finite.

Let $v \in B{[C]}$. If $v \in N^\kappa D\cap \Comp{D}$ for some $D \sim C$, then we have $d(v,\beta D) \leq \kappa$ (recall that $\kappa$ is a fixed constant). Thus it remains to consider the case $v \in D$ for all $D \sim C$.

Let $U$ be a finite subset of $B{[C]}$ such that $B{[C]} \sse G\cdot U$.
There is a constant $m \geq \kappa$ such that $d(u, \beta C) \leq m$ for  $u\in U$. We conclude that for the node
$v \in B{[C]}$ there is some $g\in G$ and $E = gC$ such that  $d(v,\beta E ) \leq m$.
Thus, we actually may assume $v \in \beta E$ and show that this implies $v \in \bigcup\limits_{D\sim C} \beta D$.

Because $C$ and $E$ are nested, we can assume (after replacing $E$ with $\Comp{E}$ if necessary) that $C\sse \Comp{E}$ or $\Comp{E}\ssnq C$. If $C\sse \Comp{E}$ (thus $E\sse \Comp{C}$), then $\beta E\sse \beta \Comp{C} \cup \Comp{C}$. But $v \in C$, hence
$v \in  \beta \Comp{C} = \beta C$. 
On the other hand, if $\Comp{E} \subsetneqq C$, then there is a $D\sim C$ such that $\Comp{E} \sse \Comp{D} \subsetneqq C$. It follows that $v \in D \cap \beta \Comp{E}  \sse D \cap ( \beta \Comp{D} \cup \Comp{D}) \sse \beta D$.
\end{proof}

A graph $\GG$ is said to have \emph{more than one end} if there is a finite set
$S \sse V(\GG)$ such that $\GG-S$ has at least two infinite connected components. 
Otherwise, it has at most \emph{one end}. 
Since we only consider connected and locally finite graphs, it follows that
$\GG$ has more than one end if and only if there exists a bi-infinite simple path $\alp$ such that 
$\cC(\alp) \neq \es$. 

The key property of blocks is that blocks cannot have more than one end: 
\begin{proposition}\label{prop:enden_in_Bc}
For $C\in \Copt$ the block $B{[C]}$ has at most one end. 
\end{proposition}

\begin{proof}
Assume by contradiction that $B{[C]}$ has more than one end.  By  \prettyref{lm:VDconBc} $B{[C]}$ is connected, hence there is a bi-infinite 
simple path $\alp$ and a finite subset $S\sse B{[C]}$ such that 
two different connected components of $B{[C]} -S$  contain 
infinitely many elements of $\alp$. However, for all $D \sim C$ we have  $\alp\sse B{[C]} \sse N^\kappa D$ and $N^\kappa D \cap \Comp{D}$ is finite. Hence for all 
$D \sim C$ almost all nodes of $\alp$ are in $D$
and $\abs{\alp \cap \Comp{D}} < \infty$.  

By \prettyref{lm:VDconBc},
there are two different connected components of $\GG- N^\ell S$ 
containing each infinitely many elements of $\alp$. Thus, the set $\cC(\alp)$ is not empty, hence there is an optimal cut $E \in \Copt(\alp)$.
This means $\abs{\alp \cap E} = \infty = \abs{\alp \cap \Comp{E}}$.
The cuts $C$ and $E$ are nested. We cannot have 
$E \sse \Comp{C}$ or $\Comp{E} \sse \Comp{C}$ because $\abs{\alpha \cap \Comp{C}}<\infty$.
Hence, by symmetry $E \subsetneqq C$. 
By \prettyref{lm:tree_set}, there is some $D\in[C]$ such that $E \sse \Comp{D}\subsetneqq C$. But we have just seen that almost all nodes of $\alp$ belong to $D$. Thus, $\abs{\alp \cap E} < \infty$. This is a contradiction. 
\end{proof}


\subsection{Actions on accessible graphs}\label{sec:hugo}

In this section  $\cG$ denotes  a class of groups which is closed under taking  normal subgroups of finite index. In our application  $\cG$ will be the class of all finite groups. But actually many other classes of groups are closed under taking finite-index normal subgroups as e.g.~the class of f.g.~virtually
free groups or e.g.~the class of finitely presented groups.

\begin{proposition}\label{prop:lea}
Let $\GG$ is a connected, locally finite, and  accessible graph such that 
$\Aut(\GG)\bs \GG$ is finite and  
let $G$ be a group acting on  $\Gamma$ such that all vertex stabilizers $G_v = \set{g\in G}{gv = v}$ belong to the 
class $\cG$. Then we have: 
\begin{enumerate}
\item The group $G$ acts with virtually $\cG$ edge stabilizers on the tree $T(\Copt)$.
\item If $B{[C]}$ is  finite for all  $C\in \Copt$, then
$G$ acts with virtually $\cG$ vertex stabilizers on the tree $T(\Copt)$.
\end{enumerate}
\end{proposition}

\begin{proof}
First, let $\es \neq U\sse V(\GG)$ be any finite set. 
The action of $G$ induces a \homo from the stabilizer
$G_U = \set{g \in G}{gU \sse U}$ to 
the finite group of permutations on $U$. Its kernel is  $\bigcap_{u\in U} G_u$. 

Now fix one vertex $v\in U$. Then for every $k\in \N$ an element $g\in G_v$ defines a permutation on the set of vertices $\set{u\in V(\Gamma)}{d(v,u)\leq k}$. Choose $k$ large enough such that $U\sse  \set{u\in V(\Gamma)}{d(v,u)\leq k}$ and let $N$ be the kernel of the map $G_v \longrightarrow \Sym{\set{u\in V(\Gamma)}{d(v,u)\leq k}}$. Then $N$ has finite index in $G_v$ because $\Gamma$ is locally finite. The action of $G$ on $V(\Gamma)$ is with $\cG$-stabilizers and $\cG$ closed under forming finite index normal subgroups, so $N$ is $\cG$.
Furthermore, $N\leq \bigcap_{u\in U} G_u\leq G_U$ with finite index and so $G_U$ is virtually $\cG$.

An element in an edge stabilizer $G_{\{C,\Comp{C}\}}$  maps $\beta C$ to itself. Since $\beta C$ is finite, $G_{\{C,\Comp{C}\}}$  is  virtually $\cG$. 

Now, let $g \in G_{[C]}$, then we have $g(B{[C]}) = B{[C]}$ by \prettyref{lm:blocks}.  If $B{[C]}$ is  finite,  $G_{[C]}$ is virtually $\cG$.
\end{proof}


\section{Finite treewidth}\label{sec:ftw}
Tree decompositions were introduced by Robertson and Seymour in connection with their famous result on graph minors, 
\cite{RobertsonS84}. For some basic properties of tree decompositions see \cite{diestel}.

\begin{dfn}
Let $\Gamma= (V(\Gamma),E(\Gamma))$ be a graph and $\cP(\GG)$ the family of subsets of 
$V(\Gamma)$. A tree $T=(V(T),E(T))$ together with a mapping  
$$V(T) \to \cP(\GG),\; t \mapsto X_t$$ 
is called tree decomposition of $\Gamma$ if the following conditions are fulfilled:
\begin{enumerate}[(T1)]
\item For every node $v\in V(\Gamma)$ there is some $t\in V(T)$ such that $v\in X_t$, i.e., $\qquad V(\Gamma) = \bigcup_{t\in V(T)} X_t$. 
\item For every edge $uv\in E(\Gamma)$ there is some $t\in V(T)$ such that  $u, v \in X_t$.
\item If $v \in X_t \cap X_s$, then we have $v \in X_r$ for all vertices $r$ 
of the tree which are on the unique geodesic path from $s$ to $t$, i.e., the set $\set{ t\in V(T)}{ v \in X_t}$ forms a subtree of $T$.
\end{enumerate}

Let $k \in \N$. 
A graph $\Gamma$ is said to have \emph{treewidth} $k$ if there exists a tree decomposition such that $\abs{X_t}\leq k+1$ for all $t \in V(T)$. We say that $\GG$ has \emph{finite treewidth} if it has treewidth $k$ for  some $k\in \N$.
The sets $X_t$ are called buckets or bags.

\end{dfn}

\begin{lemma}\label{lm:subgraph}
If $\Gamma$ has finite treewidth, then all subgraphs of $\Gamma$ have finite treewidth, too. 
\end{lemma}

\begin{proof}
 Trivial. 
\end{proof}

\begin{lemma}\label{lm:otto}
If $\Gamma$ is locally finite  of finite treewidth $k$, then there 
is a tree decomposition $T=(V(T),E(T))$ satisfying the following conditions. 
\begin{enumerate}
\item Each vertex $v\in \GG$ occurs in finitely many bags, only. 
\item We have $1 \leq \abs X \leq k$ for all bags $X$. In particular, bags are not empty. 
\item If two bags $X$ and $Y$ are connected by some edge in the tree $E(T)$, then 
$X \cap Y \neq \es$.
\item The tree $T$ is locally finite. 
\end{enumerate}
\end{lemma}

\begin{proof}
We start with a tree decomposition $T=(V(T),E(T))$ such that $\abs{X_t}\leq k$ for all $t\in V(T)$ and show that we can transform it into one meeting the desired conditions.
For every edge $uv \in E(\Gamma)$ choose and fix some vertex $t = t_{uv} \in V(T)$ with $u,v \in X_t$. Now, for each vertex $u$ let $T_u$ be the finite subtree spanned by the $t_{uv}$ for $v \in V(\Gamma)$. It is finite because $\Gamma$ is locally finite.
Remove $u$ from all bags which do not belong to $T_u$. This yields still a tree decomposition. 

Next, let $x \in X$ and $y\in Y$ where $X$ and $Y$ are two bags, and let 
$x=x_0, \ldots, x_n=y$ be some path in $\GG$ connecting $x$ and $y$. 
Let $Z$ be on the geodesic in the tree $T$ from bag $X$ to bag $Y$. 
An induction on $n$ shows that $Z \cap \smallset{x_0,\ldots, x_n} \neq \es$.
Removing all empty bags we therefore have still  a tree decomposition. 

Now, if in addition, $x \in X$, $y\in Y$ and $X$ and $Y$ are neighbors in the tree $T$, then 
we can define $i = \max\set{i}{x_i \in X}$. We have $i \geq 0$ and if 
$i = n$, then $y \in X \cap Y$. Thus, we may assume $i < n$. Looking at the location where $x_i, x_{i+1}$ are in the same bag, we see that $x_i \in Y$. 

Now, we can put things together to derive that $T$ is locally finite:
For each bag $X$ each of the neighbors contains at least one element of $X$.
But every $x$ is contained in at most finitely many bags. Hence, the result follows. 
\end{proof}

\begin{lemma}\label{lm:sep_von_1}
Let $\Gamma$ be a graph of finite treewidth and uniformly bounded degree. Then there exists some $k\in \N$ such that: For every one-sided infinite simple path $\gamma$, every  $v_0\in V(\Gamma)$, and every $n \in \N$ there is a  $k$-cut $D$ with $d(v_0, \Comp{D})\geq n$, $v_0\in D$, and $\abs{\Comp{D}\cap \gamma} = \infty$.
\end{lemma}

\begin{remark}
It follows from the following proofs that in the case of $\Gamma$ being a locally finite Cayley graph also the converse of the lemma holds. Thus, when restricting to Cayley graphs of f.g.{} groups, the statement of \prettyref{lm:sep_von_1} gives a characterization of Cayley graphs of context-free groups by its own. A very similar result is due to Woess \cite{Woess89}. Is states that a group is context-free if and only if the ends of its Cayley graph have uniformly bounded diameter.
\end{remark}

\begin{proof}
Let $d$ be the maximal degree of $\Gamma$ and let 
$m = \max\set{{\abs{X_t}}}{t \in V(T)}$ be the maximal size of a bag in the tree decomposition $(T,\mathcal{X})$. We let $k = dm$. 

Let $t_0\in V(T)$ such that $v_0\in X_{t_0}$. Consider vertices  $u, v \in V(\Gamma) - X_{t_0}$ which are in bags of two different connected components of $T- t_0$. Then every path from $u$ to $v$ has a vertex in $X_{t_0}$, so $u$ and $v$ are not in the same connected component of $\Gamma - X_{t_0}$.
Now let $C_{t_0,\gamma}$ be the connected component of $\Gamma - X_{t_0}$ which contains infinitely many vertices of $\gamma$.
Then the set $C_{t_0,\gam}$ is contained in the union of the bags of one connected component of $T- t_0$. Let $t_1$ be the neighbor of $t_0$ in this connected component, which is uniquely defined because $T$ is a tree.

Repeating this procedure yields a simple path $t_0, t_1, t_2, \ldots$ in $T$  and a sequence of connected sets
$C_{t_0,\gam}, C_{t_1,\gam}$, $C_{t_2,\gam}, \ldots$ such that $\abs{\gam \cap C_{t_i,\gam}} = \infty$ for all $i\in \N$. By \prettyref{lm:otto}, we may assume that  every node $v \in V(\Gamma)$  is contained in only finitely many bags.  Hence, we can choose $\ell$ large enough such that $X_{t_\ell}$ does not contain any $v \in V(\Gamma)$ with $d(v_0,v) \leq n$. 

Now, let $D$ be the connected component of  $\Comp{C_{t_\ell,\gam}}$ which contains $v_0$. Then $\Comp{D}$ is connected because every vertex in another connected component of $\Comp{C_{t_\ell,\gam}}$  is connected with $C_{t_\ell,\gam}$ inside of $\Comp{D}$.

Since every edge of $\delta D$ has one node in $X_{t_\ell}$, we have $\abs{\delta D} \leq dm =  k$. Thus, $D$ is a $k$-cut with  $v_0\in D$ and $\abs{\Comp{D}\cap \gamma} = \infty$. Furthermore, since every path from $v_0$ to a vertex $v\in \Comp{D}$ uses a vertex of $X_\ell$, we have $d(v_0, \Comp{D})\geq n$.

\end{proof}

\begin{proposition}\label{prop:fritz}
Let $\Gamma$ be a graph of finite treewidth and uniformly bounded degree. 
Then $\Gamma$ is accessible.
\end{proposition}

\begin{proof}
Let $\alpha$ be a bi-infinite simple path such that $ \cC(\alpha)\neq\es$
and let $C\in\cC(\alpha)$. We fix a vertex $v_0 \in \beta C$ and we let 
$n = \max\set{d(v_0,w)}{ w \in \beta C}$. Let $k \in \N$ be according to 
\prettyref{lm:sep_von_1}. It follows that there is a $k$-cut $D$  with $\abs{\alpha\cap \Comp{D}}= \infty $, $v_0\in D$, and $d(v_0, \Comp{D})\geq n$. Because of the choice of $n$, we also have $\beta C \sse D$ what means that either $C\sse D $ or  $\Comp C \sse D$. In either case  $D$ splits $\alpha$ in two infinite pieces.
\end{proof}

\begin{lemma}\label{lm:VDzweis_geod}
Let $\Gamma$ be a connected, locally finite, and infinite graph such that
$\Aut(\Gamma)\bs \Gamma$ is finite. Then there is a bi-infinite geodesic.
 \end{lemma}
\begin{proof}
 There are arbitrarily long geodesics, hence geodesics of every length.  For each geodesics $\gamma$ with an  odd number of vertices  
 let $m(\gamma)$ be the vertex in the middle. Because $\Aut(\Gamma)\setminus\Gamma$ is finite, there exists a single vertex $v_0$ such that infinitely many geodesics $\gamma$ satisfy
 $m(\gamma) = v_0$. These geodesics form the vertices of  a tree as follows: The root is $v_0$ (viewed as a geodesic of length 0). The parent of a geodesic 
 $(v_{-k}, v_{-k+1} \lds v_0 \lds v_{k-1}, v_k)$ is defined as  $(v_{-k+1} \lds v_0 \lds v_{k-1})$. Since $\Gamma$ is locally finite, we obtain an infinite tree where each node
 has finite degree. By K\"onigs Lemma there is an infinite path, which defines 
 a bi-infinite geodesic through $v_0$. 
\end{proof}

Note that we cannot remove any of the requirements in \prettyref{lm:VDzweis_geod}. 
In particular, we cannot remove that $\Aut(\Gamma)\bs \Gamma$ is finite. For example consider the graph $\Gamma$ with $V(\Gamma) = \Z$ and $E(\Gamma)=\set{(n,n\pm 1), (n,-n)}{n\in \Z}$.
This graph is connected, locally finite, and infinite. It has a bi-infinite simple path, but there is no bi-infinite geodesic.

\begin{lemma}\label{lm:two_ends}
Let $\Gamma$ be connected, locally finite, and infinite such that 
$\Aut(\Gamma)\bs \Gamma$ is finite and let $\GG$ have finite treewidth. Then $\GG$ has more than one end.
 \end{lemma}

\begin{proof}
The graph $\Gamma$ has uniformly bounded degree because it is locally finite and $\Aut(\Gamma)\bs \Gamma$ is finite. By \prettyref{lm:sep_von_1}, there is some $k$ such that for every $n\in \N$, $v_0 \in V(\Gamma)$ and every one-sided infinite simple path $\alpha$  there is a  $k$-cut $C$ with $v_0\in C$, $d(v_0,\Comp{C})\geq n$, and $\abs{\Comp{C}\cap \alpha} = \infty$.

By \prettyref{lm:VDendl_orbits_k_cuts}, 
there are only finitely many orbits of $k$-cuts under the action of $\Aut(\Gamma)$. Therefore, there is some $m\in \N$ such that $\max\set{d(u,v)}{ u,v \in \beta C}\leq m$ for all $k$-cuts $C$.

Assume that $\Gamma(G)$ has only one end. Now, by \prettyref{lm:VDzweis_geod}, there is a bi-infinite geodesic $\alpha = \dots, v_{-2},v_{-1},v_0,v_1,v_2\dots$. Let $C$ be a $k$-cut with $d(v_0, \Comp{C})> m$ such that  $v_0 \in C$  and $\abs{\alpha\cap \Comp{C}}=\infty$. Then $\abs{\alpha\cap C}<\infty$, for otherwise $\cC(\alpha)\neq \es$.

Hence, there are $i,j>m$ with $v_{-i}, v_j \in \beta C\cap \Comp{C}$. But this implies $d(v_{-i},v_j)=d(v_{-i},v_0)+d(v_0,v_j)> 2m$ in contradiction to $d(u,v)\leq m$ for all $u,v \in \beta C$.
 \end{proof}

Now we have all the tools to state and prove our main theorem.
\begin{theorem}\label{thm:new_alpha}
Let $\cG$ be a class of groups which is closed under taking finite-index  normal subgroups.
Let $\Gamma$ be a connected, locally finite graph of finite treewidth. 
Let a group $G$ act on $\Gamma$ such that $G\bs \Gamma$ is finite and each 
node stabilizer $G_v$ is in $\cG$. 

Then $G$ acts on the tree $T(\Copt)$ such that all vertex and edge 
stabilizers are virtually $\cG$ and $G\bs T(\Copt)$ is finite.
\end{theorem}

\begin{proof}
The blocks $B{[C]}$ have finite treewidth by \prettyref{lm:subgraph}.
By \prettyref{lm:endl_orbits_Bc},  $G_{[C]}$ acts with finitely 
many orbits on $B[C]$. Hence, we can apply \prettyref{lm:two_ends} what implies that the blocks are finite or have more than one end. The latter case is excluded by \prettyref{prop:enden_in_Bc}, which states that they have at most one end. That means that the blocks are finite. 
The theorem then follows with \prettyref{lm:VDendl_orbits_k_cuts} and \prettyref{prop:lea}.
\end{proof}

\begin{corollary}\label{cor:nixstall}
Let a group $G$ act on  a connected, locally finite graph $\Gamma$ of finite treewidth such that $G\bs \Gamma$ is finite and each 
node stabilizer $G_v$ is finite. 
Then  $G$ is the fundamental group of a finite graph of finite groups.
\end{corollary}

\begin{proof}
By \prettyref{thm:new_alpha}, $G$ acts on a tree $T$ with finite vertex stabilizers such that $G\bs T$ is  finite. 
Bass-Serre theory (\cite{serre80}) yields the result. 
\end{proof}

Note that if we know that $G$ is finitely generated, then the condition $\abs{G\bs \Gamma}<\infty$  in \prettyref{thm:new_alpha} and \prettyref{cor:nixstall} is no real restriction since in this case we always can construct a subgraph of $\Gamma$ on which $G$ acts with finitely many orbits. To do that we proceed as follows: Let $\Sigma$ be a finite generating set of $G$ and let $v_0\in V(\Gamma)$ be some arbitrary vertex. For all $a\in \Sigma$ we fix paths $\gamma_a$ from $v_0$ to $av_0$. Let $\Delta$ be the subgraph of $\Gamma$ induced by the vertex set $G\cdot \bigcup_{a\in \Sigma} \gamma_a$. This graph is connected, locally finite and it has finite treewidth by \prettyref{lm:subgraph}.

Another interesting observation about the tree $T(\Copt)$ is that together with the blocks $B{[C]}$ it forms a tree decomposition of $\Gamma$ of finite width.

\section{Context-free groups} 
A {\em formal language} is a subset $L$ of the free monoid $\Sg^*$ over some alphabet $\Sg$. Here, an \emph{alphabet}
 simply means any finite set. 
We say that a class $\cK$ of formal languages is closed under {\em inverse \homo} if  $L \in \cK$ implies $\psi^{-1}(L) \in \cK$
for all \homo{}s $\psi: \Sg'^* \to \Sg^*$. Almost all classes 
investigated in formal language theory or complexity theory are closed 
under inverse \homo, see e.g.{} \cite{HU}. For example, all classes 
in the Chomsky hierarchy have this property. Other examples are the 
classes of deterministic context-free languages, the class of languages 
where the  membership problem can be solved in polynomial time, and the 
class of recursive languages. 

Let $\cK$ be a class of languages which is closed under inverse \homo{}s. We say that the
\emph{word problem} of a group $G$ belongs to the 
class $\cK$ if there is \homo $\pi: \Sg^* \to G$ onto $G$ such that  $\pi^{-1}(1)\in \cK$. This is a property of $G$ and does not depend on the 
presentation $\pi: \Sg^* \to G$: Indeed, let $\pi': \Sg'^* \to G$
be another presentation of $G$. Since  $\Sg'^*$ is free, we find 
a \homo $\psi: \Sg'^* \to \Sg^*$ such that $\pi' = \pi \circ \psi$. 
Hence, $\pi'^{-1} (1) = \psi^{-1} (\pi^{-1}(1))\in \cK.$ 
For simplicity, we say that a group $G$ is \emph{context-free} if the word problem of $G$
is context-free. 
By well-known and classical results of Anisimov  it is known that all context-free groups are finitely presented \cite[Thm.~2]{anisimov72} (see also \prettyref{sec:anis}); and the word problem of a group $G$ is regular \IFF $G$ is finite \cite[Thm.~1]{anisimov71}. 
The proofs of these facts are actually very easy by using the standard ``pumping properties'' of context-free (resp.~regular) languages. 

\subsubsection{Solving the word problem using deterministic pushdown automata}\label{sec:dcf}

Let $G$ be a finitely generated virtually free group and $F(X)$ be a free 
subgroup of finite index. Choose a set $R$ with  $1 \in R \sse G$ such that 
the canonical projection
$G \to F(X)\bs G$ induces  a bijection between $R$ and the finite quotient 
$F(X)\bs G$. We use the disjoint union $\Sg = X^\pm \cup R$ as a finite
generating alphabet, where $X^\pm = X \cup X^{-1}$. For all letters $a,b\in \Sg$ we can define 
rewrite rules as follows: 
$$ ab \to x_{ab}r \quad  \text{if $x_{ab}$ is a word over $X^\pm$ and $r\in R$ such that $ab=  x_{ab}r\in G.$}$$

This system can be used by a deterministic pushdown automaton 
transforming an input word $w \in \Sg^*$ into its normal form $w =xr$ with 
$x \in (X^\pm)^*$ and $r\in R$: First, we choose $k \in \N$ such that $k \geq  \abs{x_{ab}r}$ for all rules $ab \to x_{ab}r$. The pushdown stack contains 
freely reduced words over $X^\pm$, the set of states are the words $yr \in F(X)\cdot R$ 
of length at most $k$. We start with an empty stack in state $1\in R$ and with 
the input word $w$. We perform the following instructions: 
\begin{itemize}
\item If the input is empty and the state is a letter $r\in R$, then stop. 
\item If the state is a letter $s\in R$, but the input is not empty, then read the next input letter $b$ and change the state to $x_{sb}r$ according to the rule $sb \to x_{ab}r$. 
\item If the state is a word  $ys\in F(X)\cdot R$ with $1 \neq y\in  F(X)$ and the stack content is a freely reduced word $z$ over $X^\pm$, then replace (within less than $k$ steps) $z$ by the freely reduced word corresponding the group element $zy \in F(X)$, and after that switch to  the state $s \in R$.
\end{itemize}

The description how the pushdown automaton works is just standard way how to 
compute normal forms in linear time. Indeed, if we start with an input word $w$, then we stop in a configuration where $x$ is a freely reduced word on the stack and we are in some state $r\in R$. It is clear that $w = xr \in G$. Hence, in order to solve the word problem we only have to check whether $x=1$ and $r=1$.

\subsubsection{Finitely generated virtually free groups are context-free}\label{sec:agcfg}
The statement itself follows from the precedent subsection and standard
facts how to transform a pushdown automaton into a context-free grammar, see any textbook on formal languages like \cite{HU}. 
Let us recall however that, a priori, the class of 
context-free groups could be larger than the class of deterministic context-free groups.

It is well-known that there are context-free languages which are not deterministic
context-free. Indeed, consider the group $\Z \times \Z$ with generators
$a= (1,0)$, $b= (0,1)$, and $c= (-1,-1)$. A standard exercise 
shows that set of the words $w \in \smallset{a,b,c}^*$ which are equal to 
$(0,0)$ is not context-free, but its complement is context-free. 
It cannot be deterministic context-free because deterministic 
context-free languages are closed under complementation, \cite{HU}.  
Thus, $\Z \times \Z$ is co-context-free in the sense of \cite{HoltRRT05}. 
The class of co-context-free groups is very interesting in its own, 
for example it includes the {H}igman-{T}hompson group \cite{LehnertS07}. 

\subsubsection{Context-free groups are finitely presented}\label{sec:anis}
 Anisimov \cite{anisimov72} used the so-called $uvwxy$-Theorem
 in order to show that context-free groups are finitely presented. We obtain however a more concise finite presentation 
by using the production rules of a 
context-free grammar as defining relations. To be more precise, let $\pi: \Sg^* \to G$ a surjective 
\homo such that $L_G = \set{w\in \Sg^*}{ \pi(w) = 1}$ is context-free.
Let $(V, \Sg, P, S)$ be a context-free grammar 
which generates $L_G$ according to the notation of \cite{HU}:  
This means  $V \cap \Sg = \es$ and all production rules of $P$ have the form $A\to \alp$ where
$A\in V$ is a variable and $\alp \in (V\cup \Sg)^*$ is a word. 
We may assume that every variable $A\in V$ appears in some derivation 
$$S \RAS*P \gam A \del  \RAS*P w \in \Sg^*.$$
(If there is no such derivation, we may remove $A$ from the grammar.)
Now, the canonical \homo{}s $\Sg^* \to F(\Sg) \to F(V \cup \Sg) \to  F(V \cup \Sg) /P$ yield an isomorphism:
$$G = \Sg^*/\set{u = 1} {u\in L_G} \to F(V \cup \Sg) /P.$$ 
This fact has a straightforward verification. It has been generalized to other languages and grammar types leading  to the notion of \emph{Hotz-isomorphism}. We refer to \cite{dm89rairo} for details and some open problems in this area.

\subsubsection{Quasi-isometric sections}\label{sec:qis}
This section yields a direct construction 
of  a context-free grammar (in Chomsky normal form) associated to 
a f.g.~virtually free group. Thus, we do not rely on any formal definition 
for a push-down automaton or the result that the accepted language of 
push-down automaton is always context-free. This is standard fact in formal language theory, but its proof is non-trivial. So we prefer to circumvent these constructions. 
We shall use the fact that virtually free groups have a presentation with a \emph{quasi-isometric section} as defined below. In \cite{BridsonG93} Bridson and Gilman introduced quasi-isometric sections as broomlike combings and proved that the groups with quasi-isometric sections are exactly the virtually free groups.

Throughout this section we assume that $G$ is finitely generated and 
$\pi: \Sg^* \to G$ refers to a a monoid presentation. 
This means $\Sg$ is a 
finite alphabet and $\pi$  is 
a surjective \homo. 
%
By abuse of language, we simply write $ga$ for $g\pi(a)$.
The set of words $\Sg^*$ forms a tree. The empty word 
$\ew$ is the root and a word $u$ has the children $ua$ for letters $a \in \Sg$. The geodesic distance $d(u,v)$ in the tree $\Sg^*$ yields a natural metric on $\Sg^*$. That means, we have $d(u,v) =d$ 
\IFF $d = \abs {u'} + \abs {v'}$ where $u=pu'$ and  $v=pv'$ and $p$ is the longest common prefix of $u$ and $v$. 
We are interested in sections of $\pi$ which define quasi-isometric embeddings 
of the Cayleygraph of $G$ (w.r.t.{} $\pi$) into the tree $\Sg$. 
We abbreviate this as a quasi-isometric section and use  the 
following definition. 
A \emph{quasi-isometric section} of $G$ is a mapping  $\sig: G \to \Sg^*$ such that 
\renewcommand{\labelenumi}{(\arabic{enumi})}
\begin{enumerate}
\item we have $\sig(1) = \ew$,
\item  we have  $\pi(\sig(g)) = g$ for all $g \in G$, 
\item there is some $1 \leq k \in \N$ such that  $d(\sig(g), \sig(ga)) \leq k$ for all 
$g \in G$ and $a \in \Sg$. 
\end{enumerate}
\renewcommand{\labelenumi}{\arabic{enumi}.}

Note that $\sig(G)$ yields a set of \emph{normal forms} with  $\ew \in \sig(G)$.
The important property is however that
vertices $g$, $h$ of distance $d$ in the Cayley graph of $G$ (w.r.t.{} $\pi$) have representing words of distance at most $kd$ in the tree $\Sg^*$. 

The existence of a quasi-isometric section depends only on the group $G$ and not on its
presentation $\pi: \Sg^* \to G$: Indeed, 
let  $\sig: G \to \Sg^*$ be a quasi-isometric section of $G$ and $\pi': \Sg'^* \to G$
be another monoid presentation. Then we find a \homo $\tau: \Sg^* \to \Sg'^*$
such that $\pi(w) = \pi'(\tau(w))$ for all words $w \in \Sg^*$. 
Now, the  set of normal forms $\sig(G)$ is mapped onto the 
set of normal forms $\tau(\sig(G))$ satisfying (1) and (2). Moreover, consider 
$u=pu'$ and  $v=pv'$ with $\abs {u'} + \abs {v'} \leq k$. Then there is some constant
$\ell$ (depending only on $\tau$) such that $\abs {\tau(u')} + \abs {\tau(v')} \leq k \ell$.
This shows (3) for $\tau \circ \sigma: G \to \Sg'^*$. 
Thus, we can say that $G$ has a
quasi-isometric section.


It follows from \prettyref{sec:dcf} that f.g.~virtually groups 
have quasi-isometric sections. 

Now, let $G$ have a quasi-isometric section $\sig: G \to \Sg^*$  for some monoid presentation
$\pi: \Sg^* \to G$. We let $k \geq 1$ such that  $d(\sig(g), \sig(ga)) \leq k$ for all $(g,a) \in G \times \Sg$.
We are going to define a context-free grammar for the language
$L_G = \set{w\in \Sg^*}{ \pi(w) = 1}$. The grammar will be in Chomsky normal form.
First we choose a symbol $S$ (which is outside of $G\cup \Sg^*$) as axiom, then we let 
$$V = \os{S} \cup \set{g\in G}{\abs{\sig(g)} \leq k}.$$

Thus, the set of variables consists of the axiom $S$ and a finite subset of $G$.
We have the following set $P$ of rules: 
\begin{enumerate}
\item $S \to \ew$ is the so-called \emph{$\eps$-rule} in order to produce the empty word.
\item $S \to a$ for all $a \in \Sg$ such that $\pi(a) = 1$. 
\item $S\to BC$ for all $B,C \in V\cap G$ such that $1 = BC$ in $G$.
\item $A\to BC$ for all $A,B,C \in V\cap G$ such that $A = BC$ in $G$.
\item $A\to a$ for all $A\in V\cap G$ and all $a \in \Sg$ such that $A = \pi(a)$ in $G$.\end{enumerate}

It is clear that whenever $S\RAS{*}{P}w \in \Sg^*$, then we have $\pi(w)=1$.  
Now we show the converse. 
For words $u,v\in  \Sg^*$ we denote by $[u^{-1}v]$ the group element 
$\pi(u)^{-1}\pi(v)$ in $G$. Thus, $[u^{-1}v]$ is a short hand for the expression $\pi(u)^{-1}\pi(v)$.

Now, let $w=a_1\cdots a_n$ with $a_i \in \Sigma$ and 
$\pi(w) =1$. We have to show that there is a derivation $S\RAS{*}{P} w$. 
The first two types of the rules in $P$ show that this is true if $n \leq 1$. 
Hence we may assume $n \geq 2$. 
Let us define words $u_i = a_1 \cdots a_i \in \Sg^*$ for $0 \leq i  \leq n$. 
Then we have $\pi(u_0) = \pi(u_n) = 1$.
Note that $[u_{i-1}^{-1}u_i] = \pi(a_i)$ for all $1\leq i  \leq n$. In particular, 
$[u_{i-1}^{-1}u_i] \in V \cap G$ for all $1\leq i  \leq n$ because 
$\abs{\sig(\pi(a_i))} \leq k $ for $1 \leq i  \leq n$ by the choice of $k$.
We have rules $[u_{i-1}^{-1}u_i] \to a_i$ 
for all $1\leq i  \leq n$ and it remains to show that there is some derivation
$$S \RAS*P [u_{0}^{-1}u_1] \cdots [u_{n-1}^{-1}u_n].$$

Now let $u_0, \ldots u_n$ be any sequence of words words $u_i\in \Sigma^*$ such that $n \geq 2$, 
$\pi(u_0) = \pi(u_n) = 1$ and 
$[u_{i-1}^{-1}u_i] \in V$ 
for $1\leq i \leq n$. We are going to show that this already implies that there is a derivation $S \RAS*P [u_{0}^{-1}u_1] \cdots [u_{n-1}^{-1}u_n]$.
For $n = 2 $ we have a rule 
$$S \to [u_{0}^{-1}u_{1}][u_{1}^{-1}u_{2}].$$
Hence, we may assume $n \geq 3$ and we use induction. 
As $n \geq 2$ we may choose and fix some index $m$ with 
$0 < m  < n$ such that $\abs{\sig(\pi(u_m))}$ is at least as large as any other 
$\abs{\sig(\pi(u_i))}$ for $0 \leq i  \leq n$.
It follows 
$$\abs{\sig([u_{m-1}^{-1}u_{m+1}])} \leq \max\os{\abs{\sig([u_{m-1}^{-1}u_{m}])}, \abs{\sig([u_{m}^{-1}u_{m+1}])} } \leq k.$$ 
The set $P$ includes a rule 
$$[u_{m-1}^{-1}u_{m+1}] \to [u_{m-1}^{-1}u_{m}][u_{m}^{-1}u_{m+1}].$$ 
Now we are done since by induction 
$$S \RAS*P [u_{0}^{-1}u_1] \cdots [u_{m-1}^{-1}u_{m+1}] \cdots [u_{n-1}^{-1}u_n].$$

\subsubsection{Cayley graphs of context-free groups have finite treewidth}
Muller and Schupp have shown that a Cayley graph of a context-free group has a $k$-triangulation \cite{ms83}. The definition of a $k$-triangulation is technical. We skip it here because the proof  in \cite{ms83} can also be used to 
show directly that  a Cayley graph of a context-free group has finite treewidth. 
This suffices for our purposes. 
\begin{proposition}
Let $\Gamma$ be a Cayley graph of a context-free group $G$
with respect to a finite generating set $X$. Then $\Gamma$ has finite treewidth. 
\end{proposition}

\begin{proof}If $G$ is finite, then the assertion is trivial. Hence, let $G$ be infinite. 
We may assume that $1 \notin X\sse G$.
  
The vertex set of $\GG$ is the group $G$, by $B_n$ we denote the ball of radius 
$n$ around the origin $1 \in G$. Hence $B_n= \set{g \in G}{d(1,g)\leq n}$. 
We are heading for a tree decomposition where certain finite subsets 
of $G$ become nodes in the tree. 
For $n \in \N$ we define sets $V_n$  of level $n$ such that $V_0 = \smallset{\GG-1}$ and 
$V_n = \set{C }{ C \text{ is a connected component of } \Gamma - B_{n}}$ for $n \geq 1$. This defines a tree $T$ with root $B_1$ as follows: 
\begin{align*}
V(T) &=  \set{ \beta C }{ C \in V_n\,,  n \in \N}. \\
 E(T) &= \set{\smallset{ \beta C, \beta D}}{D \sse C \in V_n\,,  D \in V_{n+1}, \, n \in \N}
 \end{align*}
The nodes are subsets of $G$, hence we can identify nodes $t \in T$ with their bags $X_t \sse G$. 
If $\{g,h\}$ is an edge in the Cayley graph $\GG$, then 
there are essentially two cases; either  $d(1,g) = n $ and $d(1,h) = n +1$ 
or $d(1,g) = d(1,h) = n +1$ for some $n$. In both cases the elements  $g$, $h$ are in some bag $\bet C$ for some $C \in V_n$ and $n \in \N$. 

It remains to show that $\abs{\bet C}$ is bounded by some constant for all $C \in V_n$, $n\in\N$. It is here where
the context-freeness comes into the play. 
We denote 
$\Sg= X \cup X^{-1}$. This  is a set of monoid generators of $G$. We let 
$L_G = \set{w \in \Sg^*}{w = 1 \in G}$ its associated group language. 
By hypothesis, $L_G$ is generated by some context-free grammar $(V,\Sg, P,S)$, and we 
may assume that it is in Chomsky normal form. This means all rules are either of the form $A \to BC$ with $A,B,C \in V$ or of the form $A\to a$ with $A\in V$ and 
$a \in \Sg^*$ such that $\abs a \leq 1$. We write $A \RAS*{P} \alp$, 
if we can derive $\alp\in (V \cup \Sg)^*$ with production rules from $P$. 
We define a constant $k \in \N$, $k \geq1$ such that 
$$k \geq \max_{A \in V}\min\set{ \abs{w}}{ A \RAS*{P} w \in \Sg^*}.$$ 

Consider $C\in V_n $ and $n \in \N$.
Let  $g,h\in \bet C$. We are going to show that $d(g,h)\leq 3k$.
For $n = 0$ we have  $\bet C = B_1$.  Hence, we may assume $n \geq 1$. 

Let $\alpha$ be a geodesic path from $1$ to $g$ with label $u \in \Sg^*$, $\gamma$ a geodesic path from $h$ to $1$ with label $w\in \Sg^*$, and $\beta$  some  path from  $g$ to $h$ with label $v \in \Sg^*$ which is entirely contained in $C$. Such a path exists since $C$ is connected.
The composition of these paths forms a closed path $\alpha \beta\gamma$ with label 
$uvw$. We have $uvw \in L_G$ and there is a derivation $S\RAS{*}{}uvw$. 
We may assume that $\abs v \geq 2$ because otherwise there is nothing to do. 

Since the grammar is in Chomsky normal form we can find  a rule $A \to BC$ and derivations as follows: 
$$S \RAS*{P} u'Aw' \RAS{}{P} u'BCw'\RAS*{P} u'v'v''w' = uvw $$
such that $B \RAS*{P} v'$, $C \RAS*{P} v''$, and  $\abs{u'} \leq  \abs{u} < \abs{u'v'} < \abs{uv } \leq\abs{u'v'v''}$. 

This yields three nodes $x\in \alpha$, $y \in \beta$, and  $z\in \gamma$ such that $d(x,y)$, $d(y,z)$, $d(x,z)\leq k$. (These three nodes correspond exactly to a triangle with endpoints $x,y,z$ in the $k$-triangulation of the closed path $\alpha \beta\gamma$ in \cite{ms83}.)

\begin{figure}[ht]
\begin{center}
\def\nodedist{5}
\begin{tikzpicture}
\node[] (OO) at (-5, 0) {};
\node[left of=OO,node distance=\nodedist] ()  {$1$};

\node[] (g) at (2, 2) {};
\node[] (h) at (2, -2) {};
\node[above of=g,node distance=\nodedist] () {$g$};
\node[below of=h,node distance=\nodedist] ()  {$h$};

\draw (2,-2) arc(-60:60:2.309401076758503058);
\node (vv) at (2+1.154700538379251529,0){};
\node[right of=vv,node distance=\nodedist] ()  {$\beta$};

\draw (-5, 0) ..controls(-3.75,0.12).. (-3,0.2)
	..controls(-1.5,0.36) and (0,1)..
node[above] (){$\alpha$} (1, 1.46)
..controls(1.5,1.69)..
(2, 2){};

\draw (-5, 0) ..controls(-3.75,0.12).. (-3,0.2)
	..controls(-2.25,0.28) and (-1.,0.6) ..
node[below] (){$\gamma$} (1, -1)
..controls(1.5,-1.4)..
(2, -2){};

\node[] (x) at (1, 1.46) {};
\node[] (z) at (1, -1) {};
\path[] (2-1.154700538379251529,0) -- ++(20:2.309401076758503058) node (y){};
\node[above of=x,node distance=\nodedist] () {$x$};
\node[below of=z,node distance=\nodedist] ()  {$z$};
\node[right of=y,node distance=\nodedist] ()  {$y$};

\draw(1, 1.46)--node[left](){$A$} (1, -1);
\draw (2-1.154700538379251529,0) + (20:2.309401076758503058)--node[above](){$C$} (1, -1);
\draw(2-1.154700538379251529,0) + (20:2.309401076758503058)--node[below](){$B$} (1, 1.46);

\end{tikzpicture}
\end{center}
\caption[]{The distance between $g$ and $h$ is bounded by $3k$.}\label{fig:diamC}
\end{figure}
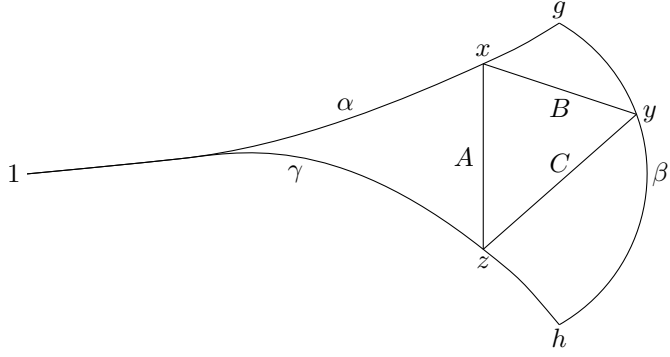

Now we have:
\[d(x,g)=d(1,g)-d(1,x) \leq d(1,y)-d(1,x)\leq d(x,y).\]
The first equality holds because $\alpha$ is geodesic and  $x$ lies on $\alpha$; the second one because $ d(1,g) \leq n+1 \leq  d(1,y)$. Likewise we obtain $d(z,h)\leq d(z,y)$.
Thus, it follows
\begin{align*}
d(g,h)&\leq d(g,x)+d(x,z)+ d(z,h)\\
&\leq d(y,x)+d(x,z)+d(z,y) \leq 3k.
\end{align*}
This implies that the size of the bags is uniformly bounded by some constant since $\Gamma$ has uniformly bounded degree.
\end{proof}

\subsection{The result of Muller and Schupp revisited}\label{sec:ms}
To date various equivalent characterizations of context-free groups are 
known. The following theorem mentions only those characterizations which we met 
in this paper for proving the fundamental result of Muller and Schupp that
context-free groups are virtually free.

\begin{theorem}\label{thm:ms}
Let $G$ be a finitely generated group and $\GG$ be its Cayley graph with respect to some finite set of generators. The following assertions are equivalent. 

\begin{enumerate}
\item\label{thm:msi} $G$ is virtually free.
\item\label{thm:msii} $G$ is deterministic context-free. 
\item\label{thm:msiii} $\GG$ has a quasi-isometric section. 
\item\label{thm:msiv} $G$ is context-free. 
\item\label{thm:msv} $\GG$ has finite treewidth. 
\item\label{thm:msvi} The group $G$ is the fundamental group of a finite graph of finite groups.
\end{enumerate}
\end{theorem}

\begin{proof}
 A review on the implications 
 \ref{thm:msi}$\implies$\ref{thm:msii}, 
 \ref{thm:msi}$\implies$\ref{thm:msiii}, 
 \ref{thm:msii}$\implies$\ref{thm:msiv}, and \ref{thm:msiii}$\implies$\ref{thm:msiv}$\implies$\ref{thm:msv} has been given  in this section.
The implication  \ref{thm:msv}$\implies$\ref{thm:msvi} is 
 a direct consequence of \prettyref{cor:nixstall}.
 The last implication \ref{thm:msvi}$\implies$\ref{thm:msi}
  follows {}from \cite{Karrass73}.
 \end{proof}

\section{Accessibility of groups}\label{sec:tw}
In this section we assume all groups to be finitely generated.
As another application of the construction in \prettyref{sec:cuts} and \prettyref{sec:blocks} we give a proof of a theorem of Thomassen and Woess \cite[Thm. 1.1]{ThomassenW93}. 
It is  an important corollary of  \cite[Thm.{} II 2.20]{DicksD89} where Dicks and Dunwoody develop their the structure tree theory.
This result allows us to consider all groups which act on a locally finite, connected, accessible graph with finite stabilizers and finitely many orbits, and not only those which act on graphs of finite treewidth. The result in \cite{ThomassenW93} gave birth to the notion of accessibility for graphs.

We need some standard facts of Bass-Serre theory. The following lemma 
is well-known, see e.g. \cite{dicks80}. For convenience of the reader, we give a proof. 

\begin{lemma}\label{lm:fg}
Let $G$ be a f.g.~fundamental group of a finite graph of groups with finite 
edge groups. Then every vertex group is finitely generated.
\end{lemma}

\begin{proof}
We give a sketch only.
 Let  $V$ be the set of vertices, $Y$ be the  set of edges
 of the finite graph, and $Z$ be the union over all edge groups. For each 
 vertex $v \in V$ let $X_v$ be some generating set of the vertex group $G_v$. 
 Then there is a finite generating set $X$ inside $\bigcup\set{X_v}{v \in V}\cup  Y \cup Z$ such that $Y \cup Z \sse X$. Now consider any $x \in X_v$, it is enough to 
 show that $x$ can be expressed as  a product over $X\cap X_v$. 
 To see this, write $x$ as shortest word in $X$. Assume this word contained a factor
 $y z y^{-1}$ with $y \in Y$ and where $z$ belongs to edge group of $y$ sitting in $G_{t(y)}$, then we could perform a ``Britton reduction''  replacing  $y z y^{-1}$ by some 
 $z'$ in the edge group of $\ov y$ sitting in $G_{s(y)}$. This would lead to a shorter word, 
 since $Y \cup Z \sse X$. Hence, this is impossible; and the word representing $x$ is ``Britton reduced''. This 
 implies that the word uses letters from $X\cap X_v$, only. 
 \end{proof}

\begin{dfn}\label{dfn:gac}
\begin{enumerate}
\item A group is called  \emph{more than one ended} (resp.{} at most one ended) if its Cayley graph has more than one end (resp.{} at most one end). (This definition does not depend on the choice of the finite generating set for the Cayley graph.) 
\item A group $G$ is called \emph{accessible} if it acts on a tree with finitely many orbits, finite edge stabilizers, and vertex stabilizers with at most one end.
\end{enumerate}
\end{dfn}

If a group $G$ is accessible, then  Bass-Serre theory yields an upper bound on 
the number how often $G$ can be split properly as an HNN-extension or amalgamated product over finite subgroups. This observation is also another definition of accessibility used frequently in literature. The link to accessibility of 
the corresponding Cayley graphs is due to the next proposition.

\begin{proposition}\label{prop:gunnar}
Let $G$ be a f.g.~group which acts on a 
tree with finitely many orbits, finite edge stabilizers and no vertex stabilizer having more than one end. Then the Cayley graph $\GG$ of $G$ 
is accessible. 
\end{proposition}

\begin{proof}Again, we give only a sketch.
Bass-Serre theory tells us that $G$ is the fundamental group of a finite graph of groups with finite edge groups. By \prettyref{lm:fg}, every vertex group $G_v$ is finitely generated. We only consider the case where
$G = A*_{H}B$ is an amalgamated product of two
f.g.~groups $A$ and $B$ over a common finite subgroup. The case of HNN-extensions follows analogously and is left to the reader.

We assume that $A$ and $B$ have accessible Cayley graphs and show that this implies that $G= A*_{H}B$ has an accessible Cayley graph. Then \prettyref{prop:gunnar} follows by induction.

Let $A$ be generated by $X_A$ and $B$ be generated by $X_B$, where $X_A$ and
$X_B$ are finite. As a generating set for $G$ we use $X = H \cup HX_AH \cup HX_BH$
and we may assume that $\GG$ is the Cayley graph of $G$ w.r.t.~$X$. We may regard the Cayley graphs of $A$ and $B$ as subgraphs of $\GG$ and refer to them as $A$ or $B$.
Now, consider any bi-infinite simple path $\alp$ in $\GG$
such that there is a cut $C$ (of finite weight) with
$\abs{C \cap \alp} =  \abs{\Comp C \cap \alp} = \infty$. We can assume that
$\abs{\del C}$ is minimal among all such cuts. %
In order to show that
$\GG$ is accessible we need a uniform bound on $\abs{\del C}$.
The path $\alp$ gives us a bi-infinite sequence of labels in $X$.
We may assume the origin $1 \in G$
is a vertex of $\alp$.
If all the labels belong to $H \cup HX_AH$, then the path is entirely in
$A$. So by hypothesis there is an upper bound on $\abs{\delta C}$.
Thus we may assume that there is at least one label in $A\sm H$ and one label in $B\sm H$ and
that $1$ is sitting between two such labels of minimal distance.
Without restriction the label on the right of $1$ belongs to $A\sm H$ and on the left it belongs to $B$.
Let $1=x_0 , x_1, x_2, \ldots$ be the one-sided infinite sequence of vertices of $\alp$ going to the right of $1$ and $\ldots, y_2, y_1, y_0 = 1$ the corresponding one on the left. For every $x\in G$ the set $HxH$ is finite.
Hence, switching to infinite subsequences  of  $x_0 , x_1, x_2, \ldots$ and $\ldots, y_2, y_1, y_0$ we may assume that no $x_{i}^{-1}x_j$ or $y_jy_{i}^{-1}$
belongs to $H$ for $i < j$. Grouping consecutive labels from $A\sm H$ (resp. $B\sm H$) into blocks we obtain  sequences
$$ h_j, \ldots, h_2, h_1 , g_1, g_2 \ldots g_i$$
such that $g_1 \in A \sm H$, $h_1 \in B \sm H$, and the $g$- and
the $h$-vertices alternate between $A \sm H$ and $ B \sm H$.
It might happen that $i$ or $j$ remains bounded, but there are sequences with
$1 \leq i,j$. The final step is to observe that
every path connecting $g_1\cdots g_i$ to $(h_j\cdots h_1)^{-1}$ must use a vertex from $H$.  This is due to the normal form theorem for amalgamated products.
\end{proof}

 Now, we can state the main result of this section. We use the notation of \prettyref{sec:blocks}.

\begin{theorem}\label{thm:new2}
Let $\Gamma$ be a locally finite, connected, accessible graph. 
Let $G$ act on $\Gamma$ such that $G\bs \Gamma$ is finite and each 
node stabilizer $G_v$ is finite. 
Then $G$ acts on the tree $T(\Copt)$ with finitely
many orbits, finite edge stabilizers, and  no vertex stabilizer has more than one end.
\end{theorem}

\begin{proof}
By \prettyref{lm:VDendl_orbits_k_cuts}  and \prettyref{prop:lea}, we know that $G$ acts with finitely many orbits and finite edge stabilizers on $T(\Copt)$.
Now consider a vertex stabilizer $G_{[C]}$ for some $C\in \Copt$. 
By \prettyref{lm:fg}, we have that  $G_{[C]}$ is finitely generated. Thus, its Cayley graph is locally finite and the number of ends is defined.

The block $B{[C]}$ has at most one end by \prettyref{prop:enden_in_Bc}. 
So it suffices to show that if the Cayley graph of $G_{[C]}$ has more than one end, then $B{[C]}$ has more than one end, too.

Because of \prettyref{lm:endl_orbits_Bc}, there is a finite set of representatives $U\sse B{[C]}$ such that $G_{[C]}~\cdot~U = B{[C]}$. 
More precisely, we can identify $B{[C]}$ with  $G_{[C]}\times U$. Let 
$Z = \set{g\in G_{[C]}}{ \exists\, u,v\in U\,:\, (u,gv) \in E(\Gamma)}$.
Then we have $\abs{Z}<\infty$ since $U$ is finite, $\Gamma$ is locally finite, and all vertex stabilizers are finite. Thus, we can define $m = \max\set{d(1,a)}{a\in Z}< \infty$ (here, $d$ denotes the distance in the Cayley graph of $G_{[C]}$).

Assume that $G_{[C]}$ has more than one end. Then the Cayley graph of $G_{[C]}$ has  a cut of finite weight $D\sse G_{[C]}$ with $\abs{D}=\abs{\Comp D} = \infty$. We claim that there are only finitely many pairs $g,h$ such that $g\in D$, $h \in G_{[C]}- D$ and $g^{-1}h \in Z$.

Indeed, since $g^{-1}h \in Z$, there is is a path of length at most $m$ from $g$ to $h$ in the Cayley graph of $G_{[C]}$. Since $g\in D$ and $h \in G_{[C]}- D$,
this path uses an edge of $\delta D$. Since $\delta D$ is finite and the Cayley graph is locally finite, there are only finitely many such paths of length at most $m$. Hence, there are only finitely many such $g$ and $h$.

Now consider $E =\set{gu}{g\in D, u\in U}\sse B{[C]}$. Every edge of the boundary $\delta E$ inside $B{[C]}$ has endpoints $gu$ and $hv$ with $g\in D$, $h\in G_{[C]} \sm {D}$ and $u,v \in U$. By the above claim there are only finitely many choices for $g$ and $h$. Together with the finiteness of $U$  this implies that $E$ has finite boundary inside $B{[C]}$.

Thus, $\beta E \sse B{[C]}$ is a finite set of vertices.
Since $\abs{D}=\abs{\Comp D} = \infty$ and $B{[C]} = G_{[C]}\times U$,  we see that $\abs{E}=\abs{B_{C} \sm E } = \infty$, too. Since $B{[C]}$ is connected,  $B{[C]} - \beta E$ has more than one infinite connected component. This in turn implies that $B{[C]}$ has more than one end. 
\end{proof}

\begin{corollary}[\cite{DicksD89},\cite{ThomassenW93}]\label{cor:new3}
A finitely generated group is accessible if and only if its Cayley graph is accessible.
\end{corollary}

\begin{proof}
If the  Cayley graph of $G$ is accessible, then \prettyref{thm:new2}  shows that the group $G$ is accessible. 
The converse is stated in \prettyref{prop:gunnar}. 
\end{proof}

\section {Conclusion}

The paper gives direct and simplified proofs for two fundamental results: 1. The Theorem of Muller and Schupp 
\cite{ms83} (context-free groups are exactly the f.g.~virtually free groups, \prettyref{thm:ms}) and 2. the accessibility result \prettyref{cor:new3} by Dicks and Dunwoody resp.~Thomassen and Woess.
This became possible due to the paper of Kr{\"o}n \cite{kroen10}. 
The intuition behind our construction is that having a Cayley graph of finite tree width should unravel a simplicial tree 
on which the group acts with finitely many orbits and finite node stabilizers. 
This intuition is worked out into a mathematical fact here. 
In particular, if we start with a context-free group $G$, the construction yields that $G$ is a 
fundamental group of a finite graph of finite groups by standard Bass-Serre theory.
By 
\cite{Karrass73}, fundamental groups of finite graphs of finite groups are f.g.~virtually free. Together this yields the proof for the difficult direction in the 
characterization of context-free groups by Muller and Schupp.


A future research program is to investigate whether our constructions can be performed effectively. The problem is to find the minimal cuts, i.e., to decide whether a given cut is minimal with respect to some bi-infinite simple path. If this could be done in elementary time for graphs of finite tree width, it would lead to an elementary time algorithm for the isomorphism problem of context-free groups by first constructing the graphs of groups and then using Krstic's algorithm (\cite{krstic89}) to check whether the fundamental groups are isomorphic.

\bibliographystyle{abbrv}
\newcommand{\Ju}{Ju}\newcommand{\Ph}{Ph}\newcommand{\Th}{Th}\newcommand{\Ch}{Ch}\newcommand{\Yu}{Yu}\newcommand{\Zh}{Zh}

\end{document}